\documentclass{amsart}
\allowdisplaybreaks[4]
\usepackage{graphicx}
\usepackage{color}
\usepackage{bbm}
\usepackage{tikz}

\newcommand{\Z}{\mathbf{Z}}

\newcommand{\sig}{\sigma}

\newcommand{{\ba}}{\bf a}
\newcommand{\ve}{\varepsilon}
\newcommand{\la}{\lambda}
\newcommand{\La}{\Lambda}
\newcommand{\ga}{\gamma}

\newcommand{\pa}{\partial}
\newcommand{\ra}{\rightarrow}
\newcommand{\Om}{\Omega}
\newcommand{\del}{\delta}

\newcommand{\na}{\nabla}
\newcommand{\cd}{\cdot}

\newcommand{\De}{\Delta}
\newcommand{\al}{\alpha}

\newcommand{\be}{\begin{equation}}
\newcommand{\ee}{\end{equation}}

\newcommand{\om}{\omega}

\newtheorem{hyp}{Hypothesis}
\newtheorem{lem}{Lemma}{\bf}{\it}
{\it}{\rm}
\newtheorem{rem}{Remark}{\it}{\rm}
{\it}{\rm}

\newtheorem{theorem}{Theorem}
\newtheorem{proposition}{Proposition}
\newtheorem{corollary}{Corollary}
\numberwithin{theorem}{section}
\numberwithin{hyp}{section}
\numberwithin{lem}{section}
\numberwithin{rem}{section}
\numberwithin{equation}{section}
\numberwithin{proposition}{section}
\numberwithin{corollary}{section}
\title{ On  regularity of averaged Green's functions in homogenization of elliptic PDE}

\author{Joseph G. Conlon  and Michael Dabkowski}

\address{(Joseph G. Conlon): University of Michigan\\ Department of Mathematics\\ Ann Arbor,
  MI 48109-1109}
\email{conlon@umich.edu}

\address{(Michael Dabkowski): University of Michigan-Dearborn \\ Department of Mathematics and Statistics\\ Dearborn,
  MI 48128}
\email{mgdabkow@umich.edu}


\begin{document}

\maketitle
\pagestyle{myheadings}
\markboth{ Regularity of Averaged Green's Functions}{Joseph G. Conlon  and Michael Dabkowski}
\begin{abstract}
This paper is concerned with establishing estimates on averaged Green's functions for a uniformly elliptic divergence form
partial difference equation with random coefficients on  the $d$ dimensional integer lattice $\mathbb{Z}^d$.  It has previously been shown
that the averaged Green's function is point-wise well approximated by the homogenized Green's function at large length scales. Corresponding results also hold
for (fractional) derivatives of the averaged Green's function up to but not including the second derivative.  Second derivative results have been established 
when the elliptic equation generates a positive semi-group, as in the diagonal case. Here the second derivative result is shown more generally by using  an insight of Bourgain. 
\end{abstract}

\maketitle
\setcounter{tocdepth}{1}
\tableofcontents

\vspace{.1in}

\section{Introduction}\label{sec:intro}
Let $(\Om,\mathcal{F},P)$ be a probability space and  denote by $\langle \  \cd \ \rangle$ expectation w.r. to the measure $P$.   We assume that the $d$ dimensional integer lattice $\mathbb{Z}^d$ acts on $\Om$  by translation operators  $\tau_x:\Om\ra\Om, \ x\in\mathbb{Z}^d$, which are measure preserving and satisfy the properties $\tau_x\tau_y=\tau_{x+y}, \ \tau_0= \ {\rm identity}, \ x,y\in\mathbb{Z}^d$. 
Consider a bounded measurable function  ${\bf a}:\Om\ra\mathbb{R}^{d(d+1)/2}$  from $\Om$ to the space of symmetric $d\times d$ matrices which satisfies the quadratic form inequality  
\begin{equation} \label{A1}
\la I_d \le {\bf a}(\om) \le \ I_d, \ \ \ \ \ \om\in\Om,
\end{equation}
where $I_d$ is the identity matrix in $d$ dimensions and $0< \la\le 1$.
We shall be interested in solutions $u(x,\eta,\om)$ to the discrete elliptic equation
\be \label{B1}
\eta u(x,\eta,\om)+\nabla^*{\bf a}(\tau_x\om)\nabla u(x,\eta,\om)=h(x), \quad x\in \mathbb{Z}^d, \ \om\in\Om.
\ee
In (\ref{B1}) we take $\eta\ge 0$ and $\nabla$ the discrete gradient operator, which has adjoint $\nabla^*$. Thus $\nabla$ is a $d$ dimensional { \it column} operator and $\nabla^*$ a $d$ dimensional {\it row} operator, which act on  functions 
$\phi:\mathbb{Z}^d\ra\mathbb{R}$ by
\begin{eqnarray} \label{C1}
\na \phi(x) &=& \big( \na_1 \phi(x),... \ \na_d\phi(x) \big), \quad  \na_i \phi(x) = \phi (x + {\bf e}_i) - \phi(x),  \\
\na^* \phi(x) &=& \big( \na^*_1 \phi(x),... \ \na^*_d\phi(x) \big), \quad  \na^*_i \phi(x) = \phi (x - {\bf e}_i) - \phi(x). \nonumber
\end{eqnarray}
In (\ref{C1}) the vector  ${\bf e}_i \in \mathbb{Z}^d$ has 1 as the ith coordinate and 0 for the other coordinates, $1\le i \le  d$. The Green's function $(x,y)\ra G_{\mathbf{a},\eta}(x,y,\om), \ x,y\in\mathbb{Z}^d,$ for (\ref{B1}) is a solution to the equation
\be \label{D1}
\eta G_{\mathbf{a},\eta}(x,y,\om),+\nabla_x^*{\bf a}(\tau_x\om)\nabla_x G_{\mathbf{a},\eta}(x,y,\om) \ = \ \del(x-y), \quad x\in \Z^d, \ \om\in\Om,
\ee
and is therefore $\mathbf{a}$-harmonic when $\eta=0$ and  $x\in\mathbb{Z}^d-\{y\}$. 
In this paper we shall be interested in studying the averaged Green's function,
\be \label{E1}
\left\langle \ G_{\mathbf{a},\eta}(x,y,\cdot) \ \right\rangle \ = \ G_{\mathbf{a},\eta}(x-y) \ . 
\ee

It is well known \cite{k,pv,zko} that if the translation operators $\tau_x, \ x\in\mathbb{Z}^d$, are ergodic  on $\Om$ then solutions to the random equation (\ref{B1}) converge to solutions of a constant coefficient equation under suitable scaling. Thus suppose $f:\mathbb{R}^d\ra\mathbb{R}$ is a $C^\infty$ function with compact support and for $\ve$ satisfying $0<\ve\le 1$ set $h(x)=\ve^2f(\ve x), \ x\in\mathbb{Z}^d$, in (\ref{B1}).  Then $u(x/\ve,\ve^2\eta,\om)$ converges with probability $1$ as $\ve\ra 0$ to a function $u_{\rm hom}(x,\eta), \ x\in\mathbb{R}^d$, which is the solution to the constant coefficient elliptic PDE
\be \label{F1}
\eta u_{\rm hom}(x,\eta)+\na^* {\bf a}_{\rm hom}\na u_{\rm hom}(x,\eta)=f(x), \quad x\in\mathbb{R}^d,
\ee
where the $d\times d$ symmetric matrix ${\bf a}_{\rm hom}$ satisfies the quadratic form inequality (\ref{A1}).  This homogenization result can be viewed as a kind of strong law of large numbers, and it is natural to enquire if there is a corresponding central limit theorem for environments which are strongly mixing. This involves estimating the rate of convergence of  $u(x/\ve,\ve^2\eta,\om)$  to $u_{\rm hom}(x,\eta)$ as $\ve\ra0$. 

Rates of convergence questions in homogenization can  be formulated in terms of the behavior of solutions to (\ref{A1}) at large length scales. In this paper we shall be concerned with estimating the difference at large length scales  between
the averaged Green's function (\ref{E1}) and the homogenized Green's function $G_{{\bf a}_{\rm hom},\eta}(\cdot)$, which is the Green's function for (\ref{F1}). It was shown in \cite{cs1} that  for certain environments  there exists $\al>0$ such that
\be \label{G1}
|G_{{\bf a},\eta}(x)-G_{{\bf a}_{\rm hom},\eta}(x)|  \le  \frac{C}{(|x|+1)^{d-2+\alpha}} e^{-\gamma\sqrt{\eta} |x|}, \    \ x\in\mathbb{Z}^d-\{0\},  
\ee
where $C,\gamma>0$ are constants. A similar result holds for the first derivative:
\be\label{H1}
 |\na G_{{\bf a},\eta}(x)-\na G_{{\bf a}_{\rm hom},\eta}(x)|  \le \frac{C}{(|x|+1)^{d-1+\alpha}} e^{-\gamma\sqrt{\eta} |x|},   \    \ x\in\mathbb{Z}^d-\{0\}.
 \ee
 The method  of proof for (\ref{G1}), (\ref{H1}) extends to bounds on fractional derivatives up to but not including the second derivative.  If one makes a further assumption that the elliptic equation (\ref{A1}) generates a positive semigroup (as in the diagonal case) one can use \cite{cs1} the methods of Nash (as implemented in this context by Delmotte and Deuschel \cite{dd})  to obtain the corresponding second derivative estimate:
 \be \label{M1}
 |\na\na^* G_{{\bf a},\eta}(x)-\na\na^* G_{{\bf a}_{\rm hom},\eta}(x)|  \le \frac{C}{(|x|+1)^{d+\alpha}} e^{-\gamma\sqrt{\eta} |x|}  \    \ x\in\mathbb{Z}^d-\{0\}. 
\ee
Here we show that (\ref{M1}) holds more generally in the case of an i.i.d. environment and also for a massless Euclidean field environment.  For the i.i.d. environment we have the following:
\begin{theorem} Suppose ${\bf a}(\cdot)$ satisfies (\ref{A1}),  the matrices ${\bf a}(\tau_x\cdot), \ x\in\Z^d$, are independent, and $0<\eta\le1$. Then for $d\ge 2$ there exists $\alpha>0$ depending only on $d$ and  $\la>0$ in (\ref{A1})  such that  (\ref{M1}) holds for some positive constants $\ga, C$, depending only on $d$ and  $\la$. 
\end{theorem}
The proof of Theorem 1.1 is based on an insight of Bourgain \cite{b}, which was further developed by Kim and Lemm \cite{kl}.  Corollary 1.5 of \cite{kl} obtains bounds on derivatives of $G_{{\bf a},\eta}(\cdot)$ 
up to $d+1$ derivatives {\it provided} $1-\la$ is sufficiently small. In contrast, Theorem 1.1 holds for all $\la, \ 0<\la\le1$. In \cite{b} Bourgain makes subtle use of Poisson kernels and Chebyshev polynomials (the disjointification lemma)   to make certain estimates. In $\S5$ we replace these  by more conventional arguments, using the covariance formula (\ref{B5})-(\ref{D5}).  In the appendix we discuss the main results of Bourgain and the improvements by  Kim-Lemm. The method of Kim-Lemm is 
essentially an iteration of Bourgain's method.  We  illustrate  graphically why this is possible and why no further improvement can be carried out. The reason is that the Bourgain and Kim-Lemm improvement correspond to tree graphs, whereas a further iteration of Bourgain's method would introduce a cycle into the graph. 

We also consider here probability spaces $(\Om,\mathcal{F},P)$ on certain Euclidean fields. These Euclidean fields are determined by a potential $V : \mathbb{R}^d \ra \mathbb{R}$ which is a $C^2$ uniformly convex function.  Thus the second derivative ${\bf a}(\cdot)=V''(\cdot)$ of $V(\cdot)$ is assumed to satisfy the inequality (\ref{A1}). Next consider functions 
$\phi : \mathbb{Z}^d \ra \mathbb{R}$ on the integer lattice in $\mathbb{R}^d$.  Let $\Om$ be the space of
all such functions and $\mathcal{F}$ be the Borel algebra generated by finite 
dimensional rectangles 
$\{ \phi \in \Om: \  |\phi(x_i) - a_i| < r_i, \ i=1,...,N\}$, 
$x_i \in \mathbb{Z}^d, \ a_i \in \mathbb{R}, \ r_i > 0, \ i=1,...,N, \ N \ge 1$.  The translation operators $\tau_x:\Om\ra\Om, \ x\in\mathbb{Z}^d$, are defined by $\tau_x\phi(z)=\phi(x+z), \ z\in\mathbb{Z}^d$. For any $d\ge 1$ and $m>0$ one can define \cite{c1,fs} a unique ergodic translation invariant probability 
measure $P$ on $(\Om, \mathcal{F})$ which depends on the function $V$ and $m$.  The 
measure is formally given as
\begin{equation} \label{N1}
\exp \left[ - \sum_{x\in \Z^d} V\left( \na\phi(x)\right)+m^2\phi(x)^2 \right] \prod_{x\in \Z^d} d\phi(x)/{\rm normalization}.
\end{equation}

The limit as $m\ra 0$ of the measure (\ref{N1}) is a probability measure on gradient fields $\om:\mathbb{Z}^d\ra\mathbb{R}^d$, where formally $\om(x)=\na(\phi(x)), \ x\in\mathbb{Z}^d$. This massless field measure is ergodic with respect to translation operators \cite{c1,fs} for all $d\ge 1$. In the case $d=1$ it has a simple structure since  then the variables $\om(x), \ x\in\mathbb{Z}$, are i.i.d.  Note that in the probability space $(\Om,\mathcal{F},P)$ for the massless field,  the Borel algebra $\mathcal{F}$ is  generated by the intersection of finite dimensional rectangles and the hyperplanes imposing the gradient constraints for $\om(\cdot)$.  For $d\ge 3$ the gradient field  measure induces a measure on fields $\phi:\mathbb{Z}^d\ra\mathbb{R}$ which is simply the limit of the measures (\ref{N1})  as $m\ra 0$.   For $d=1,2$ the $m\ra 0$ limit of the measures (\ref{N1}) on fields  $\phi:\mathbb{Z}^d\ra\mathbb{R}$ does not exist.   However the limit of the gradient field measure does exist. 
\begin{theorem} Let $\tilde{{\bf a}}:\mathbb{R}^d\ra\mathbb{R}^{d(d+1)/2}$  be a $C^1$ function on $\mathbb{R}^d$ with values in the space of symmetric $d\times d$ matrices which satisfy the quadratic form inequality  (\ref{A1}). Let $(\Om, \mathcal{F}, P)$ be the probability space of  gradient fields $\om(\cdot)=\na\phi(\cdot)$ determined by the limit of (\ref{N1}) as $m\ra 0$, and set ${\bf a}(\cdot)$ in (\ref{B1}) to be ${\bf a}(\om)=\tilde{{\bf a}}(\om(0)), \ \om\in\Om$.  
Suppose in addition that the derivative $D\tilde{{\bf a}}(\cdot)$ of $\tilde{{\bf a}}(\cdot)$ satisfies  the inequality $\|D\tilde{{\bf a}}(\cdot)\|_{\infty}\le \La$.  Then for $d\ge 2$ there exists $\alpha>0$ depending only on $d$ and  $\la$, such that (\ref{M1}) holds for some positive constants $\ga$ and $C=C_1[\La/\sqrt{\la}+1]$, where $\ga,C_1$ depend only on $d$ and $\la$. 
\end{theorem}
It was shown by Duerinckx, Lemm and Pagano \cite{dlp} that Theorem 2.1 holds for a Gaussian measure (\ref{N1}) i.e. $V(\cdot)$ quadratic, provided $1-\la$ is sufficiently small. Parallel to the earlier results of Kim and Lemm \cite{kl} for the i.i.d. case, they also establish  (Corollary 3.8 of \cite{dlp}) estimates corresponding to (\ref{N1}) for derivatives beyond the second derivative estimate.

\subsection*{Notation}
There are various parameters and notations used  throughout the paper, which we summarize here:
\begin{itemize}
\item  The parameters $\la,\eta$ which are associated with the uniform convexity of the function $\mathbf{a}(\cdot)$ in (\ref{A1}) and the elliptic equation (\ref{B1}). 
The parameter $\la$ also arises as a bound on the second derivative of the function $V(\cdot)$ in  (\ref{N1}). The parameter $\eta$ occurs in (\ref{C3}) in the definition of the Green's function.
\smallskip\item  We denote by $\mathbb{T}^d$ the $d$ dimensional torus, so  $\mathbb{T}^d\equiv[-\pi,\pi]^d$. A  function $f: \mathbb{T}^d\ra\mathbb{C}$ is represented as a periodic function $f: [-\pi,\pi]^d\ra\mathbb{C}$.  The function $f(\cdot)$ may be represented in terms of its Fourier transform $\hat{f}:\mathbb{Z}^d\ra\mathbb{C}$ by the series
\be \label{O1}
f(\xi) \ = \ \sum_{x\in\mathbb{Z}^d} \hat{f}(x)e^{ix\cdot\xi} \ , \quad \xi\in[-\pi,\pi]^d \ .
\ee
\smallskip\item For a real valued function $\phi(\cdot)$ on a periodic box in $\mathbb{Z}^d$, we denote the finite difference gradient by the column vector  $\na\phi(\cdot)$ as in (\ref{C1}), and the row operator adjoint of $\na$ by $\na^*$.
The negative finite difference Laplacian $-\De$ is given by $-\De=\na^*\na$. 
\smallskip\item Gradients of real valued functions  $\phi(\cdot)$ are denoted by $\om(\cdot)$.  This notation applies both when the function is deterministic and random.
\smallskip\item A vector $v\in\mathbb{C}^d$ is considered to be a column vector,  with its adjoint being the corresponding row vector $v^*$. The  Euclidean inner product  on a  complex Hilbert space $\mathcal{H}$ is denoted by $[\cdot,\cdot]_{\mathcal{H}}$. 
 The  sesquilinear form $[v_2,v_1]\ra [v_2,v_1]_{\mathcal{H}}$  is linear in $v_1$ and sesquilinear in $v_2$.
 \smallskip\item For $1\le p<\infty$ the space $\ell_p(\mathbb{Z}^d)$  denotes a space of $p$ summable functions on  $\mathbb{Z}^d$ with norm $\|\cdot\|_p$. We also use weighted spaces $\ell_{p,k}(\mathbb{Z}^d)$ with norm $\|\cdot\|_{p,k}$  as defined by (\ref{L5}). In $\S3$ we consider spaces $L^p(\mathbb{T}^d)$ of $p$ integrable functions on $\mathbb{T}^d$ and the corresponding weak spaces  $L_w^p(\mathbb{T}^d)$ -see (\ref{X3})-(\ref{AA3}).
\end{itemize}

\vspace{.1in}
\section{Tools for proofs}
We extend here the methods used in \cite{cs1} by including the important insight of Bourgain \cite{b}, which is Lemma 1 of his paper.  Our starting point as in \cite{cs1}  is the representation (\ref{A4}). Then one concentrates on establishing
properties of the function $\xi\ra q(\xi,\eta), \ \xi\in\mathbb{T}^d, \ 0<\eta\le 1,$  which are uniform in $\eta$ and imply theorems 1.1,1.2. It was already shown in \cite{cn2} that derivatives $\xi\ra (\pa/\pa \xi_1)^r q(\xi,\eta), \ \xi=[\xi_1,\dots,\xi_d]\in\mathbb{T}^d,$ are in  weak spaces $L^p_w(\mathbb{T}^d)$ with $p=d/r$, provided $r<d$. Here $r$ does not have to be an integer. Fractional derivatives of $q(\cdot,\eta)$ are defined in  the statement of Corollary 3.1. This property of $q(\cdot,\eta)$ is almost sufficient to prove theorems 1.1,1.2, but not quite. In \cite{cs1} these theorems were proved by using the additional fact that the function  $x\ra \na\na^*G_{{\bf a},\eta}(x), \ x\in\mathbb{Z}^d,$  is H\"{o}lder continuous (see Corollary 4.1). Delmotte and Deuschel \cite{dd} established the H\"{o}lder continuity in the case when the elliptic equation (\ref{B1}) generates a positive semigroup.  Bourgain's method enables one to  bypass this assumption.

Using Bourgain's method, we are able to show that the Fourier transform  $x\ra \hat{q}(x,\eta), \ x\in\mathbb{Z}^d,$  of $q(\cdot,\cdot)$ satisfies an inequality $|\hat{q}(x,\eta)|\le C/[1+|x|]^{d+\beta}, \ x\in\mathbb{Z}^d, $ for some $\beta>0$ (see Lemma 4.3). 
This is a stronger inequality than what can be obtained using the known properties of  derivatives of the function $\xi\ra q(\xi,\eta), \ \xi\in\mathbb{T}^d$, which merely imply that  $|\hat{q}(x,\eta)|\le C_\nu/[1+|x|]^{d-\nu}, \ x\in\mathbb{Z}^d, $ for any $\nu>0$. 
The main point in the paper is that one can take advantage of the stronger inequality for $\hat{q}(\cdot,\eta)$ by using Parseval's theorem-see equation (\ref{AS4}).

\vspace{.1in}
\section{Fourier space representation and estimates}
We shall be concerned with studying properties of the function $q(\xi,\eta), \ \xi\in\mathbb{R}^d, \ \eta>0$, which is periodic in $\mathbb{R}^d$  with fundamental domain $\mathbb{T}^d$. It is defined in terms of an operator $T_{\xi,\eta}$ on $L^2(\Om;\mathbb{C}^d)$ and the $d\times d$ matrix valued function 
$\mathbf{b}(\om)=I_d-\mathbf{a}(\om), \ \om\in\Om$.  Since $\la I_d\le \mathbf{a}(\cdot)\le I_d$ in the quadratic form sense, it follows that $0\le \mathbf{b}(\cdot)\le (1-\la)I_d$.  We have from (2.13) of \cite{cs1} that
\be \label{A3}
T_{\xi,\eta} \ = \ f(\xi-A)\  ,  \quad\quad  f(\zeta) \ = \ \frac{e(\zeta)e^*(\zeta)}{\eta+e(\zeta)^*e(\zeta)} \ ,  \quad \zeta \in\mathbb{R}^d \ ,
\ee
where $A$ is the generator of the translation group $\tau_x, \ x\in\mathbb{Z}^d$, so $\tau_x=\exp[iA\cdot x]$. We may write $T_{\xi,\eta}$ in configuration space representation as
\be \label{B3}
T_{\xi,\eta} g(\om) \ = \ \sum_{x\in\mathbb{Z}^d} \na\na^*G_\eta(x)e^{ix\cdot\xi}g(\tau_{-x}\om) \ , \quad g(\cdot)\in L^2(\Om;\mathbb{C}^d) \ ,
\ee
where $G_\eta:\mathbb{Z}^d\ra\mathbb{R}$  is the Green's function which satisfies the equation,
\be \label{C3}
\eta G_\eta(x)+\na^*\na G_\eta(x) \ = \ \del(x), \quad x\in\mathbb{Z}^d \ .
\ee
We have from (2.19), (2.20) of \cite{cs1} that
\begin{multline} \label{D3}
q(\xi,\eta) \ = \ \left\langle \  \mathbf{a}(\cdot) \  \right\rangle  -\sum_{m=1}^\infty h_m(\xi,\eta) \ , \quad
h_m(\xi,\eta)   \ = \ \left\langle \  \mathbf{b}(\cdot)[PT_{\xi,\eta}\mathbf{b}(\cdot)]^m \  \right\rangle  \  ,
\end{multline}
where $P$ is the projection on $L^2(\Om;\mathbb{C}^d)$  orthogonal to the constant.

 In this section we shall obtain estimates on the functions $\xi\ra h_m(\xi,\eta), \ \xi\in\mathbb{R}^d$, and their analytic continuations into $\mathbb{C}^d$. 
Towards this end we first obtain estimates on the operator $T_{\xi,\eta}$. 
One easily sees from the representation (\ref{A3})  that  $T_{\xi,\eta}$  is a bounded operator on $L^2(\Om;\mathbb{C}^d)$ with norm satisfying $\|T_{\xi,\eta}\|\le 1$.  We give an alternative proof of  Lemma 2.1 of \cite{cs1}
based on this representation. 
\begin{lem}
For fixed $\eta$ satisfying  $0<\eta\le 1$,  the $C^\infty$ operator value function  $\xi\ra T_{\xi,\eta}$  from $\mathbb{R}^d$ to the space of bounded linear operators $\mathcal{L}[L^2(\Om;\mathbb{C}^d)]$
on $L^2(\Om;\mathbb{C}^d)$  has an analytic continuation to a region  $\{\xi\in\mathbb{C}^d:  |\Im\xi|^2<c_d\eta\}$ for some constant $c_d>0$ depending only on $d$.  Furthermore,  for $\xi$ in this region there exists a constant 
$C_d$ depending only on $d$ such that
\be \label{E3}
\|T_{\xi,\eta}\|_{L^2(\Om;\mathbb{C}^d)} \ \le \ 1+\frac{C_d|\Im\xi|^2}{\eta} \ , \ \quad \|T_{\xi,\eta}-T_{\Re\xi,\eta}\|_{L^2(\Om;\mathbb{C}^d)} \ \le \ \frac{C_d|\Im\xi|}{\sqrt{\eta}} \ .
\ee
\end{lem}
\begin{proof}
Since the operator $A=[A_1,\dots, A_d]$ in  (\ref{A3})  has the property that the $A_j, \ j=1,\dots,d$ form  a commuting set of self-adjoint operators, it is sufficient to prove the result for $A\equiv 0$. 
In that case  $T_{\xi,\eta}$ is  for $\xi\in\mathbb{R}^d$ a Hermitian matrix operator  acting on $\mathbb{C}^d$. 
The matrix $T_{\xi,\eta}$ has entries $ T_{\xi,\eta}(j,k), \ 1\le j,k\le d,$ where
\begin{multline} \label{F3}
T_{\xi,\eta}(j,k) \ = \ \frac{e_j(\xi)e_k(\xi)^*}{\eta+e(\xi)^*e(\xi)} \ = \ \frac{ \{e^{-i\xi_j}-1\}\{e^{i\xi_k}-1\}}{\eta+ \sum_{r=1}^d \{e^{i\xi_r}-1\}\{e^{-i\xi_r}-1\} }  \\
= \ \frac{ \{e^{-i\Re\xi_j+\Im\xi_j}-1\}\{e^{i\Re\xi_k-\Im\xi_k}-1\}}{\eta+ \sum_{r=1}^d \{e^{i\Re\xi_r-\Im\xi_r}-1\}\{e^{-i\Re\xi_r+\Im\xi_r}-1\} }  \ .
\end{multline} 
The denominator in the final expression of (\ref{F3}) may be written as
\begin{multline} \label{G3}
{\rm denom} \ = \ \eta+ e(\Re\xi)^*e(\Re\xi)+ \sum_{r=1}^d \{1-e^{\Im\xi_r}\}\{1-e^{-\Im\xi_r}\} \\  
+ \sum_{r=1}^d \{1-e^{\Im\xi_r}\}e_r(\Re\xi)  +\sum_{r=1}^d \{1-e^{-\Im\xi_r}\}e_r(\Re\xi)^* \ .
\end{multline}
We write the numerator in (\ref{F3}) as 
\begin{multline} \label{H3}
{\rm num}_{j,k} \ = \ e_j(\Re\xi)e_k(\Re\xi)^*+e^{-i\Re\xi_j+i\Re\xi_k}\{e^{\Im\xi_j}-1\}\{e^{-\Im\xi_k}-1\} \\
+e^{-i\Re\xi_j}\{e^{\Im\xi_j}-1\}e_k(\Re\xi)^*+e_j(\Re\xi)e^{i\Re\xi_k}\{e^{-\Im\xi_k}-1\}  \ .
\end{multline}
Then we have that
\be \label{I3}
| {\rm denom}| \ \ge   \ \frac{\eta}{2}+\left[1-\frac{C_1|\Im\xi|^2}{\eta}\right]  |e(\Re\xi)|^2 \ , \quad {\rm for \ some \ constant \ } C_1.
\ee
From (\ref{H3}) we also have for the $d\times d$  matrix ${\rm num}_{j,k}, \ 1\le j,k\le d,$ acting on $\mathbb{C}^d$   that for some constants $C_2,C_3$ one has the inequality
\be \label{J3}
|{\rm num}  (v)| \le \ \left[   |e(\Re\xi)|^2  +C_2|\Im\xi|^2+ C_3|e(\Re\xi)||\Im\xi|                    \right]|v| \ , \quad v\in\mathbb{C}^d \ .
\ee
On using the inequality 
\be \label{K3}
C_3|e(\Re\xi)||\Im\xi|  \ \le \ \frac{C_3^2|\Im\xi|^2}{\eta}|e(\Re\xi)|^2+ \frac{\eta}{4} \ ,
\ee
in (\ref{J3}) we see that (\ref{I3}), (\ref{J3}) implies the first inequality of (\ref{E3}). For the second inequality we use instead of (\ref{K3}) the inequality 
\be \label{L3}
|e(\Re\xi)||\Im\xi|  \ \le \  \left(\frac{|\Im\xi|}{\sqrt{\eta}}\right) \sqrt{\eta} |e(\Re\xi)| \ ,
\ee 
whence the inequality follows from (\ref{I3}). 
\end{proof}
We consider operators taking periodic  functions $f(\cdot)$ on $\mathbb{R}^d$ with fundamental domain $\mathbb{T}^d$ and  range  $\mathcal{L}(\mathbb{C}^d)$, the $d\times  d$ matrices on $\mathbb{C}^d$, to periodic  functions $T_{r,\eta}f(\cdot)$ on  $\mathbb{R}^d$ with range $\mathcal{L}[L^2(\Om;\mathbb{C}^d)]$, the bounded linear operators on  $L^2(\Om;\mathbb{C}^d)$. For $r=0,1,\dots, \ \eta>0,$ the operators $S_{r,\eta},T_{r,\eta}$ are defined by
\be \label{M3}
\begin{array}{ccc}
S_{r,\eta} f(\xi) \ &=& \  f(\xi-A)\mathbf{b}(\cdot) [T_{\xi,\eta}\mathbf{b}(\cdot)]^r    \ , \quad  \xi\in\mathbb{R}^d \ , \\
T_{r,\eta} f(\xi) \ &=& \  f(\xi-A)P\mathbf{b}(\cdot) [T_{\xi,\eta}P\mathbf{b}(\cdot)]^r    \ , \quad  \xi\in\mathbb{R}^d \ . 
\end{array}
\ee
We may also consider  $S_{r,\eta},T_{r,\eta}$ as having range  $\mathcal{L}[\mathbb{C}^d;L^2(\Om;\mathbb{C}^d)]$, the linear operators from $\mathbb{C}^d$ to $L^2(\Om;\mathbb{C}^d)$. 
In that case  $S_{r,\eta}f(\xi)$ can be represented as a random $d\times d$ matrix   $S_{r,\eta}f(\xi,\om), \ \om\in\Om$, and similarly $T_{r,\eta}f(\xi)$. 
We may write $S_{r,\eta}f(\xi,\om), \ \xi\in\mathbb{T}^d,$  in configuration space representation as
\begin{multline} \label{N3}
S_{r,\eta} f(\xi,\om)  \ = \  \sum_{y\in\mathbb{Z}^d} \hat{S}_{r,\eta} f(y,\om)e^{iy\cdot\xi}  \  , \quad {\rm where} \\
\hat{S}_{r,\eta} f(y_r,\om) \ = \ \sum_{y_0,\dots, y_{r-1}\in\mathbb{Z}^d} \hat{f}(y_0)\mathbf{b}(\tau_{-y_0}\om) \prod_{j=1}^r  \na\na^*G_\eta(y_j-y_{j-1})\mathbf{b}(\tau_{-y_j}\om)  \ . 
\end{multline}
The random matrices $S_{r,\eta}f(\xi,\cdot),T_{r,\eta}f(\xi,\cdot) \in\mathcal{L}[\mathbb{C}^d;L^2(\Om;\mathbb{C}^d)]$ are related by the identity,
\begin{multline} \label{O3}
T_{r,\eta}f(\xi,\om)  \ = \ S_{r,\eta} f(\xi,\om)-\sum_{k=0}^r f(\xi-A)[\mathbf{b}(\cdot)T_{\xi,\eta}]^k(\om) \left\langle \  \mathbf{b}(\cdot)[T_{\xi,\eta} P\mathbf{b}(\cdot)]^{r-k}  \   \right\rangle         \\
= \   S_{r,\eta} f(\xi,\om)-f(\xi)\left\langle \  \mathbf{b}(\cdot)[T_{\xi,\eta} P\mathbf{b}(\cdot)]^r  \   \right\rangle \\
 -\sum_{k=1}^r S_{k-1,\eta}f(\xi,\om))\frac{e(\xi)e^*(\xi)}{\eta+e(\xi)^*e(\xi)}  \left\langle \  \mathbf{b}(\cdot)[T_{\xi,\eta} P\mathbf{b}(\cdot)]^{r-k}  \   \right\rangle  \ .
\end{multline}
For $v_1,v_2\in\mathbb{C}^d$, we have that
\begin{multline} \label{P3}
 v_1^*\left\langle \  \mathbf{b}(\cdot)\{T_{\xi,\eta} P\mathbf{b}(\cdot)\}^r  \   \right\rangle v_2 \ =  \  \Big[      \mathbf{b}(\cdot)v_1,  \       \{T_{\xi,\eta} P\mathbf{b}(\cdot)\}^r v_2               \Big]_{L^2(\Om;\mathbb{C}^d)}  \\
 = \ \Big[             \{T^*_{\xi,\eta} P\mathbf{b}(\cdot)\}^r v_1,  \mathbf{b}(\cdot)v_2,                \Big]_{L^2(\Om;\mathbb{C}^d)}  \ , \quad {\rm where \ } T^*_{\xi,\eta} \ {\rm is  \ the \  adjoint \ of \ } T_{\xi,\eta} \ .
\end{multline}
Since $T_{\xi,\eta}$  is self-adjoint we have that $ \{T^*_{\xi,\eta} P\mathbf{b}(\cdot)\}^r = T_{r-1,\eta} f(\xi)$, where $f(\cdot)$ is given by (\ref{A3}).

We shall show that the operators $S_{r,\eta}$ are bounded on various spaces and obtain formulas for the bounds. The identities (\ref{O3}), (\ref{P3}) will then enable us to obtain similar bounds for the operators $T_{r,\eta}$. 
\begin{lem}
Assume $[\xi,\eta]$ satisfy the conditions of Lemma 3.1 and that the function $f:\mathbb{T}^d\ra\mathcal{L}(\mathbb{C}^d)$ may be analytically continued to the region
  $\{\xi\in\mathbb{C}^d:  |\Im\xi|^2<c_d\eta\}$.   If $a\in\mathbb{R}^d$ satisfies $|a|^2<c_d\eta$,  the  following inequalities hold:
\be \label{Q3}
\sup_{\xi\in \mathbb{T}^d, \ v\in\mathbb{C}^d:|v|=1} \left\|        S_{r,\eta} f(\xi+ai)v             \right\|_{L^2(\Om;\mathbb{C}^d)}  \ \le \ (1-\la)^{r+1}\left[     1+\frac{C_d|a|^2}{\eta}     \right]^r \sup_{\xi\in\mathbb{T}^d, \ v\in\mathbb{C}^d:|v|=1}|f(\xi+ai)v| \  , 
\ee
\begin{multline} \label{R3}
 \left[\frac{1}{(2\pi)^d} \int_{\mathbb{T}^d} d\xi \  \left\|       v^* S_{r,\eta} f(\xi+ia)            \right\|^2_{L^2(\Om;\mathbb{C}^d)}  \ \right]^{1/2} \\
 \le \ (1-\la)^{r+1}\left[     1+\frac{C_d|a|^2}{\eta}     \right]^r    \left[  \frac{1}{(2\pi)^d}\int_{\mathbb{T}^d} d\xi   \ |v^*f(\xi+ia)|^2  \ \right]^{1/2}  \ , \quad v\in\mathbb{C}^d \ .
\end{multline} 
\end{lem}
\begin{proof}
The inequality (\ref{Q3}) follows immediately from Lemma 3.1, by using the bound on $\mathbf{b}(\cdot)$.  To prove (\ref{R3}) we use the method of proof in Lemma 3.4 of \cite{cn2}.   The main point is that the integral on the  LHS of (\ref{R3}) is the same as
\be \label{S3}
\int_\Om d\om\  \sum_{y_r\in\mathbb{Z}^d}\left| \ e^{-a\cdot y_r} v^*\hat{S}_{r,\eta} f(y_r,\om) \ \right|^2 \ ,
\ee
 where $\hat{S}_{r,\eta}$ is defined in (\ref{N3}). 
 We see using the formula for $ \hat{S}_{r,\eta}$ in  (\ref{N3}) that
\begin{multline} \label{T3}
 \left[ \sum_{y_r\in\mathbb{Z}^d} \left| e^{-a\cdot y_r} v^*\hat{S}_{r,\eta} f(y_r,\om)\right|^2 \right]^{1/2} \ \le \\
 (1-\la)^{r+1}\left\{     1+\frac{C_d|a|^2}{\eta}     \right\}^r \left[\sum_{y\in\mathbb{Z}^d} \left|v^*\hat{f}(y)e^{-a\cdot y}\right|^2\right]^{1/2} \\
  = \  (1-\la)^{r+1}\left\{     1+\frac{C_d|a|^2}{\eta}     \right\}^r \ \left[\frac{1}{(2\pi)^d}\int_{\mathbb{T}^d} d\xi   \  |v^*f(\xi+ia)|^2\right]^{1/2} \ .
\end{multline}
The inequality (\ref{R3}) follows from (\ref{S3}), (\ref{T3}). 
\end{proof}
\begin{lem}
Assume $[\xi,\eta]$ satisfy the conditions of Lemma 3.1 and that the function $f:\mathbb{T}^d\ra\mathcal{L}(\mathbb{C}^d)$ may be analytically continued to the region
  $\{\xi\in\mathbb{C}^d:  |\Im\xi|^2<c_d\eta\}$.   If $a\in\mathbb{R}^d$ satisfies $|a|^2<c_d\eta$,  the  following inequalities hold:
\be \label{U3}
\sup_{\xi\in \mathbb{T}^d, \ v\in\mathbb{C}^d:|v|=1} \left\|        T_{r,\eta} f(\xi+ai)v             \right\|_{L^2(\Om;\mathbb{C}^d)}  \ \le \ (1-\la)^{r+1}\left[     1+\frac{C_d|a|^2}{\eta}     \right]^r \sup_{\xi\in\mathbb{T}^d, \ v\in\mathbb{C}^d:|v|=1}|f(\xi+ai)v| \  , 
\ee
\begin{multline} \label{V3}
 \left[\frac{1}{(2\pi)^d} \int_{\mathbb{T}^d} d\xi \  \left\|      v^*  T_{r,\eta} f(\xi+ia)            \right\|^2_{L^2(\Om;\mathbb{C}^d)}  \ \right]^{1/2} \\
 \le \ (r+2)(1-\la)^{r+1}\left[     1+\frac{C_d|a|^2}{\eta}     \right]^r    \left[  \frac{1}{(2\pi)^d}\int_{\mathbb{T}^d} d\xi   |v^*f(\xi+ia)|^2   \ \right]^{1/2}  \ , \quad v\in\mathbb{C}^d \ .
\end{multline} 
\end{lem}
\begin{proof}
The inequality (\ref{U3}) follows immediately from Lemma 3.1, by using the bound on $\mathbf{b}(\cdot)$ and the fact that the projection operator $P$ is a contraction on $L^2(\Om;\mathbb{C}^d)$.  
To prove (\ref{V3}) we use the identities (\ref{O3}), (\ref{P3}) and Lemma 3.2. It is sufficient to show that (\ref{R3}) holds when  $ S_{r,\eta} f(\xi+ia)$ is replaced by each of the $r+2$ terms on the RHS of (\ref{O3}). To see this we note that
$T^*_{\xi+ai,\eta}=T_{\xi-ai,\eta}$.  Hence we have from (\ref{P3})  that
\begin{multline} \label{W3}
\left|v^*S_{k-1,\eta}f(\xi+ai,\om))\frac{e(\xi+ai)e^*(\xi+ai)}{\eta+e(\xi+ai)^*e(\xi+ai)}  \left\langle \  \mathbf{b}(\cdot)[T_{\xi+ai,\eta} P\mathbf{b}(\cdot)]^{r-k} \  \right\rangle   \right| \\
 \le \  (1-\la)^{r-k+1}\left[     1+\frac{C_d|a|^2}{\eta}     \right]^{r-k+1} \left|v^*S_{k-1,\eta}f(\xi+ai,\om))\right| \ .
\end{multline}
We then obtain the inequality (\ref{R3})  for the $(k+2)$th term on the RHS of (\ref{O3}) from (\ref{W3}) and Lemma 3.2.
\end{proof}
For a matrix $A\in\mathcal{L}(\mathbb{C}^d)$, its Hilbert-Schmidt norm $\|A\|_{ \rm HS}$ is given by the formula  $\|A\|_{ \rm HS}^2={\rm Tr}\{A^*A\}$.  Similarly we may define the Hilbert-Schmidt norm of an operator 
$\mathcal{A}\in\mathcal{L}[\mathbb{C}^d;L^2(\Om;\mathbb{C}^d)]$, the linear mappings from  $\mathbb{C}^d$ to $L^2(\Om;\mathbb{C}^d)$, 
by  $\|\mathcal{A}(\cdot)\|_{ \rm HS}^2=\left\langle \ {\rm Tr}\{\mathcal{A}(\cdot)^*\mathcal{A(\cdot)}\} \ \right\rangle$. We can using these norms define  $L^p$ spaces of functions with domain $\mathbb{T}^d$ and ranges 
$\mathcal{L}(\mathbb{C}^d)$ and $\mathcal{L}(\mathbb{C}^d;L^2(\Om;\mathbb{C}^d))$. Thus we set
\be \label{X3}
\begin{array}{ccc}
\|A\|_p^p  \ &=& \ \frac{1}{(2\pi)^{d}}\int_{\mathbb{T}^d} \|A(\xi\|_{\rm HS}^p \ d\xi \ , \quad {\rm for} \  A:\mathbb{T}^d\ra \mathcal{L}(\mathbb{C}^d) \ , \\
\|\mathcal{A}\|_p^p  \ &=& \ \frac{1}{(2\pi)^{d}}\int_{\mathbb{T}^d} \|\mathcal{A}(\xi\|_{\rm HS}^p \ d\xi \ , \quad {\rm for} \  \mathcal{A}:\mathbb{T}^d\ra  \mathcal{L}(\mathbb{C}^d;L^2(\Om;\mathbb{C}^d)) \ .
\end{array}
\ee
Lemma 3.2 may be stated in terms of the norms (\ref{X3}) as
\be \label{Y3}
\begin{array}{ccc}
\left\|S_{r,\eta} f(\cdot+ai)\right\|_\infty \ &\le& \ \sqrt{d}(1-\la)^{r+1}\left[     1+\frac{C_d|a|^2}{\eta}     \right]^r  \|f(\cdot+ai)\|_\infty \ , \\
\left\|S_{r,\eta} f(\cdot+ai)\right\|_2 \ &\le& \ (1-\la)^{r+1}\left[     1+\frac{C_d|a|^2}{\eta}     \right]^r  \|f(\cdot+ai)\|_2 \ .
\end{array}
\ee
It follows from (\ref{Y3}) and the Riesz-Thorin interpolation theorem  \cite{rs,sw} that
\be \label{Z3}
\left\|S_{r,\eta} f(\cdot+ai)\right\|_p \ \le \ d^{1/2-1/p}(1-\la)^{r+1}\left[     1+\frac{C_d|a|^2}{\eta}     \right]^r  \|f(\cdot+ai)\|_p \ , \quad {\rm for \ } 2\le p\le \infty \ .
\ee

We have shown that  $S_{r,\eta}$ is a bounded operator from $L^p(\mathbb{T}^d;  \mathcal{L}(\mathbb{C}^d) )$ to  $L^p(\mathbb{T}^d;  \mathcal{L}(\mathbb{C}^d; L^2(\Om;\mathbb{C}^d))$  with norm  not more than  the coefficient on the RHS of (\ref{Z3}). 
We may also define weak $L^p_w$ spaces using the norms (\ref{X3}).  Thus $A(\cdot)$ is in $L^p_w( \mathbb{T}^d; \mathcal{L}(\mathbb{C}^d))$ if there is a constant $C$ such that
\be \label{AA3}
{\rm meas}\{\xi\in\mathbb{T}^d:  \|A(\xi)\|_{\rm HS}>\mu\} \ \le \ \frac{C^p}{\mu^p} \ , \quad \mu>0 \ .
\ee
The weak norm $\|A(\cdot)\|_{p,w}$ is the minimum $C$ for which (\ref{AA3}) holds. We may similarly define the space  $L_w^p(\mathbb{T}^d;  \mathcal{L}(\mathbb{C}^d; L^2(\Om;\mathbb{C}^d))$.  From (\ref{AA3})  and Hunt's interpolation theorem \cite{rs,sw}, it follows that   when $2<p<\infty$  the operator $S_{r,\eta}$  from $L_w^p(\mathbb{T}^d;  \mathcal{L}(\mathbb{C}^d) )$ to  $L_w^p(\mathbb{T}^d;  \mathcal{L}(\mathbb{C}^d; L^2(\Om;\mathbb{C}^d))$  is bounded.  
\begin{lem} {\rm (Hunt interpolation theorem)}
Let $(\mathcal{M},\mu)$ be a measure space and $ \mathcal{H}_1, \mathcal{H}_2$   Hilbert spaces.  Suppose that  $p_1,p_2,q_1,q_2\in\mathbb{R}$ satisfy $1\le p_1<p_2\le\infty, \ 1\le q_1< q_2\le \infty$, and  $T$ is a  bounded linear operator  from
$L^{p}(\mathcal{M};\mathcal{H}_1)$  to  $L^{q}(\mathcal{M};\mathcal{H}_2)$ when $p=p_1, q=q_1$ and $p=p_2, q=q_2$.   Hence there are constants $C_1,C_2$ such that 
\be \label{AB3}
\|Tf\|_{q_1}\le C_1\|f\|_{p_1} \ , \quad \|Tf\|_{q_2}\le C_2\|f\|_{p_2}  \ , \quad {\rm all \ } f:\mathcal{M}\ra\mathcal{H}_1 \ .
\ee
Then for any $p,q$ satisfying the identities
\be \label{AC3}
\frac{1}{p} \ = \ \frac{\theta}{p_1}+\frac{1-\theta}{p_2} \ ,  \quad \frac{1}{q} \ = \ \frac{\theta}{q_1}+\frac{1-\theta}{q_2} \ ,  \quad {\rm with \ } 0<\theta<1 \ ,
\ee
one has the inequality
\be \label{AD3}
\|Tf\|_{q,w}\le M C_1^\theta C_2^{1-\theta}\|f\|_{p,w} \ \ , \quad {\rm where \ } M \ {\rm depends \ only \ on \ } p_1,p_2,q_1,q_2,\theta \ .
\ee
\end{lem}
It follows from (\ref{Y3}) and  Lemma  3.4 that
\begin{multline} \label{AK3}
\left\|S_{r,\eta} f(\cdot+ai)\right\|_{p,w} \ \le \ Md^{1/2-1/p}(1-\la)^{r+1}\left[     1+\frac{C_d|a|^2}{\eta}     \right]^r  \|f(\cdot+ai)\|_{p,w} \ , \quad {\rm for \ } 2< p< \infty \ , \\
 {\rm where \ } M\ {\rm depends \  only\   on \ } p \ .
\end{multline}
 We shall also obtain  multilinear versions of (\ref{Z3}), (\ref{AK3}). To prove these, we need in addition to the Riesz-Thorin and Hunt interpolation theorems, the Marcinkiewicz interpolation theorem:
 \begin{lem} {\rm (Marcinkiewicz interpolation theorem)}
Let $(\mathcal{M},\mu)$ be a measure space and $ \mathcal{H}_1, \mathcal{H}_2$   Hilbert spaces.  Suppose that  $p_1,p_2,q_1,q_2\in\mathbb{R}$ satisfy $1\le p_1<p_2\le\infty, \ 1\le q_1< q_2\le \infty$, and  $T$ is a  bounded linear operator  from
$L^{p}(\mathcal{M};\mathcal{H}_1)$  to  $L^{q,w}(\mathcal{M};\mathcal{H}_2)$ when $p=p_1, q=q_1$ and $p=p_2, q=q_2$.   Hence there are constants $C_1,C_2$ such that 
\be \label{AE3}
\|Tf\|_{q_1,w}\le C_1\|f\|_{p_1} \ , \quad \|Tf\|_{q_2,w}\le C_2\|f\|_{p_2}  \ , \quad {\rm all \ } f:\mathcal{M}\ra\mathcal{H}_1 \ .
\ee
Then for any $p,q$ satisfying the identities
\be \label{AF3}
\frac{1}{p} \ = \ \frac{\theta}{p_1}+\frac{1-\theta}{p_2} \ ,  \quad \frac{1}{q} \ = \ \frac{\theta}{q_1}+\frac{1-\theta}{q_2} \ ,  \quad {\rm with \ } 0<\theta<1 \ ,
\ee
one has the inequality
\be \label{AG3}
\|Tf\|_{q}\le M C_1^\theta C_2^{1-\theta}\|f\|_{p} \ \ , \quad {\rm where \ } M \ {\rm depends \ only \ on \ } p_1,p_2,q_1,q_2,\theta \ .
\ee
\end{lem}
  To state the multilinear version of (\ref{Z3}) we first note that since the operator $S_{r,\eta}$ of (\ref{M3})  has range  $\mathcal{L}[L^2(\Om;\mathbb{C}^d)]$, the  product 
$\prod_{j=1}^n  S_{r_j\eta} f_j(\xi)$ is also in  $\mathcal{L}[L^2(\Om;\mathbb{C}^d)]$ and therefore in $\mathcal{L}[\mathbb{C}^d;L^2(\Om;\mathbb{C}^d)]$.
\begin{lem}
Assume $[\xi,\eta]$ satisfy the conditions of Lemma 3.1 and that  functions $f_1,\dots, f_n:\mathbb{T}^d\ra\mathcal{L}(\mathbb{C}^d)$ may be analytically continued to the region
  $\{\xi\in\mathbb{C}^d:  |\Im\xi|^2<c_d\eta\}$.  Let $r_1,\dots,r_n\in\mathbb{Z}$ be non-negative and set $r=r_1+\cdots+r_n$.  Assume $p_1,\dots p_n$ satisfy $2\le p_1,\dots, p_n\le\infty$ and $1/p_1+\cdots+1/p_n=1/p\le 1/2$.
Then  if $a\in\mathbb{R}^d$ satisfies $|a|^2<c_d\eta$,  the  following inequality holds:
\begin{multline} \label{AL3}
 \left\|             \left\{\prod_{j=1}^n  S_{r_j,\eta} f_j\right\} (\cdot+ai)          \right\|_p \\
 \le \ d^{1/2}(1-\la)^{r+n}\left[     1+\frac{C_d|a|^2}{\eta}     \right]^r    \prod_{j=1}^n \|f_j(\cdot+ai)\|_{p_j} \ ,   
\end{multline} 
where the norms in (\ref{AL3}) are given by (\ref{X3}).
\end{lem}
\begin{proof}
We proceed by induction on $k, \ 1\le k\le n$, where $1/p_1+\cdots+1/p_n=1/p\le 1/2$ and $1/p_{k+1}=\cdots 1/p_n=0$.  Following the argument of Lemma 3.2 and using the Riesz-Thorin theorem as in (\ref{Z3}), we see that  (\ref{AL3}) holds when $k=1$. 
Assume now that (\ref{AL3}) holds for some $k<n$, so for $1/p_1+\cdots+1/p_k\le 1/2$ and  $1/p_{k+1}=\cdots 1/p_n=0$.  We show that it also holds when $1/p_1+\cdots+1/p_{k+1}\le  1/2$ and $1/p_{k+2}=\cdots 1/p_n=0$.

We first consider the case when  $1/p_1+\cdots+1/p_{k+1}=  1/2$ and $1/p_{k+2}=\cdots 1/p_n=0$.
We argue as in Lemma 3.2 by  going to the configuration representation as in (\ref{S3}).   Thus we have that
\begin{multline} \label{AM3}
\left[\frac{1}{(2\pi)^d} \int_{\mathbb{T}^d} d\xi  \     \left\|      v^*  \left\{\prod_{j=1}^n  S_{r_j,\eta} f_j\right\}(\xi+ai)        \right\|^2_{L^2(\Om;\mathbb{C}^d)}\right]^{1/2} \\
 \le \  (1-\la)^{r_{k+1}+\cdots+r_n+n-k}\left[     1+\frac{C_d|a|^2}{\eta}     \right]^{r_{k+1}+\cdots+r_n}    \prod_{j=k+2}^n \|f_j(\cdot+ai)\|_{\infty}  \\
 \times \left[\int_\Om d\om \ \sum_{y\in\mathbb{Z}^d} \left|\sum_{y'\in\mathbb{Z}^d} v^*  \left\{\prod_{j=1}^k  \hat{S}_{r_j,\eta} f_j\right\} (y',\om) \hat{f}_{k+1}(y-y')e^{-a\cdot y}  \right|^2 \right]^{1/2}\ .
\end{multline}
 In (\ref{AM3}) the product $\left\{\prod_{j=1}^k  \hat{S}_{r_j,\eta} f_j\right\} (y',\om), \ y'\in\mathbb{Z}^d,$ denotes the configuration space representation of the function  $\xi\ra \prod_{j=1}^k  S_{r_j,\eta} f_j(\xi), \ \xi\in\mathbb{T}^d$, 
 with range $\mathcal{L}(\mathbb{C}^d;L^2(\Om;\mathbb{C}^d))$. By the Parseval theorem we have that
 \begin{multline} \label{AN3}
 \sum_{y\in\mathbb{Z}^d} \left|\sum_{y'\in\mathbb{Z}^d} v^*  \left\{\prod_{j=1}^k  \hat{S}_{r_j,\eta} f_j\right\} (y',\om) \hat{f}_{k+1}(y-y')e^{-a\cdot y}  \right|^2  \ = \ \\
 \frac{1}{(2\pi)^d} \int_{\mathbb{T}^d} d\xi  \left| v^*\left\{\prod_{j=1}^k  S_{r_j,\eta} f_j\right\} (\xi+ai,\om) f_{k+1}(\xi+ai)  \right|^2  \\
 \le \  \frac{1}{(2\pi)^d} \int_{\mathbb{T}^d} d\xi  \left| v^*\left\{\prod_{j=1}^k  S_{r_j,\eta} f_j\right\} (\xi+ai,\om) \right|^2 \left\|f_{k+1}(\xi+ai)  \right\|^2_{\rm HS} \ .
 \end{multline}
 It follows from (\ref{AN3}) and H\"{o}lder's inequality that
 \begin{multline} \label{AO3}
\sum_{m=1}^d \int_\Om d\om \sum_{y\in\mathbb{Z}^d} \left|\sum_{y'\in\mathbb{Z}^d} \mathbf{e}_m^*  \left\{\prod_{j=1}^k  \hat{S}_{r_j,\eta} f_j\right\} (y',\om) \hat{f}_{k+1}(y-y')e^{-a\cdot y}  \right|^2   \\
 \le \  \frac{1}{(2\pi)^d} \int_{\mathbb{T}^d} d\xi  \left\|\left\{\prod_{j=1}^k  S_{r_j,\eta} f_j\right\} (\xi+ai) \right\|_{\rm HS}^2 \left\|f_{k+1}(\xi+ai)  \right\|^2_{\rm HS}  \\
  \le \  \left\|\left\{\prod_{j=1}^k  S_{r_j,\eta} f_j\right\} (\cdot+ai) \right\|_{q_{k+1}}^2 \left\|f_{k+1}(\cdot+ai)  \right\|^2_{p_{k+1}} \ , \quad {\rm where \ } \frac{1}{p_{k+1}}+\frac{1}{q_{k+1}} \ = \ \frac{1}{2} \ .
 \end{multline}
 Since $1/q_{k+1}=1/p_1+\cdots+1/p_k\le 1/2$, we have by the  induction hypothesis that
 \begin{multline} \label{AP3}
 \left\|\left\{\prod_{j=1}^k  S_{r_j,\eta} f_j\right\} (\cdot+ai) \right\|_{q_{k+1}} \ \le \  
 d^{1/2}(1-\la)^{r_1+\cdots+r_k+k}\left[     1+\frac{C_d|a|^2}{\eta}     \right]^{r_1+\cdots r_k}    \prod_{j=1}^k \|f_j(\cdot+ai)\|_{p_j} \ .
 \end{multline}
 We conclude from (\ref{AM3})-(\ref{AP3}) that (\ref{AL3}) holds.
 
 We now have  that (\ref{AL3}) holds when  $1/p_1+\cdots+1/p_k\le 1/2, \ 1/p_{k+2}=\cdots 1/p_n=0$ and $1/p_{k+1}$ satisfies $1/p_{k+1}=0$ or $1/p_1+\cdots+1/p_{k+1}= 1/2$.  Applying the Riesz-Thorin theorem again,
 we conclude that (\ref{AL3}) holds when $1/p_1+\cdots+1/p_{k+1}\le 1/2, \ 1/p_{k+2}=\cdots 1/p_n=0$.  The result follows now by induction.
\end{proof}
\begin{lem}
Under the same conditions as in Lemma 3.6, the following inequality holds: 
\begin{multline} \label{AQ3}
 \left\|             \left\{\prod_{j=1}^n  T_{r_j,\eta} f_j\right\}(\cdot+ai)            \right\|_p \\
 \le \ d^{n/2}(1-\la)^{r+n}\left[     1+\frac{C_d|a|^2}{\eta}     \right]^r   \prod_{j=1}^n (r_j+2)\|f_j(\cdot+ai)\|_{p_j} \ .  
\end{multline} 
\end{lem}
\begin{proof}
We proceed by induction, assuming that (\ref{AQ3}) holds for $n-1$ functions $f_j$ and showing that it also holds for $n$ functions $f_j$.   We make use of an identity which generalizes (\ref{O3}). Thus we have that
\begin{multline} \label{AR3}
 \left\{\prod_{j=1}^n  T_{r_j,\eta} f_j\right\}(\xi,\om)    \ =  \  \left\{\prod_{j=1}^n  S_{r_j,\eta} f_j\right\}(\xi,\om) \\
  -f_1(\xi-A)\sum_{k=0}^{r_1-1}[\mathbf{b}(\cdot)T_{\xi,\eta}]^k(\om)\left\langle \  \mathbf{b}(\cdot)[T_{\xi,\eta} P\mathbf{b}(\cdot)]^{r_1-k} \left\{\prod_{j=2}^n  T_{r_j,\eta} f_j\right\}(\xi,\cdot)  \   \right\rangle  \\
  - \sum_{m=1}^{n-2} \left\{\prod_{j=1}^{m}f_j(\xi-A)[\mathbf{b}(\cdot)T_{\xi,\eta}]^{r_j}\right\}f_{m+1}(\xi-A) \\
  \times\sum_{k=0}^{r_{m+1}-1}[\mathbf{b}(\cdot)T_{\xi,\eta}]^k(\om)\left\langle \  \mathbf{b}(\cdot)[T_{\xi,\eta} P\mathbf{b}(\cdot)]^{r_{m+1}-k} \left\{\prod_{j=m+2}^n  T_{r_j,\eta} f_j\right\}(\xi,\cdot)  \   \right\rangle 
  \\ 
  - \left\{\prod_{j=1}^{n-1}f_j(\xi-A)[\mathbf{b}(\cdot)T_{\xi,\eta}]^{r_j}\right\} f_n(\xi-A)\sum_{k=0}^{r_n}[\mathbf{b}(\cdot)T_{\xi,\eta}]^k(\om)\left\langle \  \mathbf{b}(\cdot)[T_{\xi,\eta} P\mathbf{b}(\cdot)]^{r_n-k} \   \right\rangle  \ .
\end{multline}
Evidently Lemma 3.6 yields a bound for the first term on the RHS of (\ref{AR3}).  

We consider the second term. We have as in (\ref{P3}) that  for $v_1,v_2\in\mathbb{C}^d$, 
\begin{multline} \label{AS3}
v_1^*\left\langle \  \mathbf{b}(\cdot)[T_{\xi,\eta} P\mathbf{b}(\cdot)]^{r_1-k} \left\{\prod_{j=2}^n  T_{r_j,\eta} f_j\right\}(\xi,\cdot)  \   \right\rangle  v_2 \\
 = \  \left[              \mathbf{b}(\cdot)    \{T_{\xi,\eta} P\mathbf{b}(\cdot)\}^{r_1-k}v_1, \  \left\{\prod_{j=2}^n  T_{r_j,\eta} f_j\right\}(\xi,\cdot)   v_2            \right]_{L^2(\Om;\mathbb{C}^d)} \ , \quad \xi\in\mathbb{R}^d \ .
\end{multline}
We see from the Schwarz inequality and Lemma 3.1  that the analytic continuation to  $\xi+ai$ of the LHS of (\ref{AS3}) is bounded by
\be \label{AT3}
(1-\la)^{r_1+1-k}\left[     1+\frac{C_d|a|^2}{\eta}     \right]^{r_1-k}|v_1| \left\|               \left\{\prod_{j=2}^n  T_{r_j,\eta} f_j\right\}(\xi,\cdot)   v_2                    \right\|_{L^2(\Om;\mathbb{C}^d)} \ .
\ee
Applying vectors $v_1,v_2$ to the left and right hand side of the second term on the right in (\ref{AR3}), we obtain using (\ref{AT3}) the inequality,
\begin{multline} \label{AU3}
\left|   v_1^*G(\xi+ai,\om)v_2           \right|  \\
\le \ (1-\la)^{r_1+1-k}\left[     1+\frac{C_d|a|^2}{\eta}     \right]^{r_1+1-k}|v_1*S_{k-1,\eta} f_1(\xi+ai,\om)| \left\|               \left\{\prod_{j=2}^n  T_{r_j,\eta} f_j\right\}(\xi+ai,\cdot)   v_2                    \right\|_{L^2(\Om;\mathbb{C}^d)}  \ , \\
G(\xi,\om) \ = \ f_1(\xi-A)[\mathbf{b}(\cdot)T_{\xi,\eta}]^k(\om)\left\langle \  \mathbf{b}(\cdot)[T_{\xi,\eta} P\mathbf{b}(\cdot)]^{r_1-k} \left\{\prod_{j=2}^n  T_{r_j,\eta} f_j\right\}(\xi,\cdot)  \   \right\rangle  \ .
\end{multline} 
Noting that $\sup_{|v_1|=1} |v_1^*S_{k-1,\eta} f_1(\xi+ai,\om)| \le \|S_{k-1,\eta} f_1(\xi+ai,\om)\|_{\rm HS}$,  we see from (\ref{AU3}) that
\begin{multline} \label{AV3}
\|G(\xi+ai,\om) \|^2_{\rm HS} \ \le \ (1-\la)^{2(r_1+1-k)}\left[     1+\frac{C_d|a|^2}{\eta}     \right]^{2(r_1+1-k)}\ \\
\times \|S_{k-1,\eta} f_1(\xi+ai,\om)\|^2_{\rm HS}
\left\|               \left\{\prod_{j=2}^n  T_{r_j,\eta} f_j\right\}(\xi+ai,\cdot)                      \right\|^2_{\rm HS}  \ .
\end{multline}
Integrating (\ref{AV3}) over $\om\in\Om,$ we conclude that
\begin{multline} \label{AW3}
\|G(\xi+ai,\cdot) \|_{\rm HS} \ \le \ (1-\la)^{r_1+1-k}\left[     1+\frac{C_d|a|^2}{\eta}     \right]^{r_1+1-k} \\
\times \|S_{k-1,\eta} f_1(\xi+ai,\cdot)\|_{\rm HS}
\left\|               \left\{\prod_{j=2}^n  T_{r_j,\eta} f_j\right\}(\xi,\cdot)                      \right\|_{\rm HS}  \ .
\end{multline}
Integrating  the $p$th power of (\ref{AW3}) over $\xi\in\mathbb{T}^d$ and applying H\"{o}lder's inequality, we obtain the inequality
\begin{multline} \label{AX3}
\|G(\cdot+ai) \|_p \ \le \  (1-\la)^{r_1+1-k}\left[     1+\frac{C_d|a|^2}{\eta}     \right]^{r_1+1-k} \\
\times   \|S_{k-1,\eta} f_1(\cdot+ai)\|_{p_1} \left\|               \left\{\prod_{j=2}^n  T_{r_j,\eta} f_j\right\}(\cdot+ai)                      \right\|_{q_1} \ , \quad {\rm where \ } \frac{1}{p_1}+\frac{1}{q_1}=\frac{1}{p} \ .
\end{multline}
Observe that $1/q_1=1/p_2+\cdots+1/p_{k+1}\le 1/2$. Hence by the induction hypothesis we have from (\ref{AQ3}) that
\begin{multline} \label{AY3}
 \left\|             \left\{\prod_{j=2}^n  T_{r_j,\eta} f_j\right\}(\cdot+ai)            \right\|_{q_1} \\
 \le \ d^{(n-1)/2}(1-\la)^{r-r_1+n-1}\left[     1+\frac{C_d|a|^2}{\eta}     \right]^{r-r_1}   \prod_{j=2}^n (r_j+2)\|f_j(\cdot+ai)\|_{p_j} \ .  
\end{multline} 
It follows from (\ref{AX3}), (\ref{AY3}) and  (\ref{Z3})  that
\begin{multline} \label{AZ3}
\|G(\cdot+ai) \|_p \ \le \  d^{n/2}(1-\la)^{r+n}\left[     1+\frac{C_d|a|^2}{\eta}     \right]^{r} \\
\times   \| f_1(\cdot+ai)\|_{p_1}   \prod_{j=2}^n (r_j+2)\|f_j(\cdot+ai)\|_{p_j} \ .          
\end{multline}
Summing the inequality (\ref{AZ3})  over $k$ in (\ref{AU3}) with $0\le k\le r_1-1$, we conclude that the $p$ norm -as defined by (\ref{X3})- of the second term on the RHS of (\ref{AR3}) is bounded by
\be \label{BA3}
d^{n/2}(1-\la)^{r+n}\left[     1+\frac{C_d|a|^2}{\eta}     \right]^{r} \\
   r_1\| f_1(\cdot+ai)\|_{p_1}   \prod_{j=2}^n (r_j+2)\|f_j(\cdot+ai)\|_{p_j} \ .
\ee

The remaining terms on the RHS of (\ref{AR3}) may be estimated in a similar way. Thus, using Lemma 3.6 and the induction hypothesis,  the $p$ norm of the $m$th term in the sum over $0\le m\le n-2$ is bounded by
\begin{multline} \label{BB3}
d^{(n-m)/2}(1-\la)^{r+n}\left[     1+\frac{C_d|a|^2}{\eta}     \right]^{r} \\
   \times  \prod_{j=1}^m \|f_j(\cdot+ai)\|_{p_j}  r_{m+1}\| f_{m+1}(\cdot+ai)\|_{p_{m+1}}   \prod_{j=m+2}^n (r_j+2)\|f_j(\cdot+ai)\|_{p_j} \ .
\end{multline}
The $p$ norm of the final term on the RHS of  (\ref{AR3}) is bounded by
\be \label{BC3}
d^{1/2}(1-\la)^{r+n}\left[     1+\frac{C_d|a|^2}{\eta}     \right]^{r}  \left\{\prod_{j=1}^{n-1} \|f_j(\cdot+ai)\|_{p_j}\right\}  (r_{n}+1)\| f_{n}(\cdot+ai)\|_{p_{m+1}}   \ .
\ee

We have assumed that the inequality (\ref{AQ3}) holds for $n-1$ functions $f_j$.  Since
\be \label{BD3}
1+\sum_{m=0}^{n-2}r_{m+1}\prod_{j=m+2}^n (r_j+2)+r_n+1 \ \le \ \prod_{j=1}^n (r_j+2) \ ,
\ee
we conclude from (\ref{BA3})-(\ref{BC3}) that (\ref{AQ3}) also holds for $n$ functions $f_j$, whence the result follows by induction.
\end{proof}
We can easily obtain  a version of Lemma 3.7 corresponding to weak norms by invoking the Hunt and Marcinkiewicz interpolation theorems:   
\begin{lem}
Under the same conditions as in Lemma 3.6 and assuming $\sup_{1\le j\le n} p_j<\infty, \ p>2$, the following inequality holds: 
\begin{multline} \label{BE3}
 \left\|             \left\{\prod_{j=1}^n  T_{r_j,\eta} f_j\right\}(\cdot+ai)            \right\|_{p,w} \\
 \le \ Md^{n/2}(1-\la)^{r+n}\left[     1+\frac{C_d|a|^2}{\eta}     \right]^r   \prod_{j=1}^n (r_j+2)\|f_j(\cdot+ai)\|_{p_j,w} \ ,
\end{multline} 
where $M$ depends only on   $p_1,\dots,p_n$.  
\end{lem}
\begin{proof}
We prove (\ref{BE3}) using induction. Thus we assume that for some $k$ with $1\le k<n$ one has
\begin{multline} \label{BF3}
 \left\|             \left\{\prod_{j=1}^n  T_{r_j,\eta} f_j\right\}(\cdot+ai)            \right\|_{p,w} \\
 \le \ Md^{n/2}(1-\la)^{r+n}\left[     1+\frac{C_d|a|^2}{\eta}     \right]^r   \prod_{j=1}^k (r_j+2)\|f_j(\cdot+ai)\|_{p_j,w}  \prod_{j=k+1}^n (r_j+2)\|f_j(\cdot+ai)\|_{p_j} \ ,
\end{multline}
for all $p_1,\dots, p_n$ satisfying $\sup_{1\le j\le k} p_j<\infty$ and $p>2$, where the constant $M$ depends only on $p_1,\dots, p_n$. To see  that  (\ref{BF3}) holds when $k=1$, we consider $p_2,\dots,p_n$ to be fixed and satisfying
$1/p_2+\cdots+1/p_n=\al<1/2$.  Now we use the fact that (\ref{AQ3}) holds when $1/p_1=0, \ 1/p=\al$, and $1/p_1=1/2-\al, \ 1/p=1/2$. It follows then from  the Hunt interpolation theorem, Lemma 3.5, that (\ref{BF3}) holds for $k=1$ with constant $M$ depending only on $1/p_1,\al$.

In order to show that if  (\ref{BF3}) holds for some $k=m<n$, it also holds for $k=m+1$, we need also the Marcinkiewicz theorem, Lemma 3.6.  We apply Lemma 3.6 to the inequality (\ref{BF3}) when $k=m$  for the function $f_{m+1}$ to  obtain the inequality, 
\begin{multline} \label{BG3}
 \left\|             \left\{\prod_{j=1}^n  T_{r_j,\eta} f_j\right\}(\cdot+ai)            \right\|_{p} \\
 \le \ Md^{n/2}(1-\la)^{r+n}\left[     1+\frac{C_d|a|^2}{\eta}     \right]^r   \prod_{j=1}^m (r_j+2)\|f_j(\cdot+ai)\|_{p_j,w}  \prod_{j=m+1}^n (r_j+2)\|f_j(\cdot+ai)\|_{p_j} \ .
\end{multline}
The constant $M$ in (\ref{BG3}) differs from the constant $M$ in (\ref{BF3}) when $k=m$ since it includes now the constant in the inequality (\ref{AG3}).  However it continues to depend only on $p_1,\dots, p_n$. 
Next we apply the Hunt interpolation theorem, Lemma 3.5, to the inequality (\ref{BG3})  for the function $f_{m+1}$.  This shows that (\ref{BF3}) holds when $k=m+1$. 
\end{proof}
We apply the previous results to obtain bounds on derivatives of the function $\xi\ra h_m(\xi+ai,\eta), \ \xi\in\mathbb{T}^d,$ of (\ref{D3}).  Parallel to Proposition 2.2 of \cite{cs1} we have:
\begin{proposition}
For fixed $\eta$ with $0<\eta\le 1$, and $a\in\mathbb{R}^d$ satisfying $|a|^2<c_d\eta$ as in Lemma 3.1,  the derivatives of the function  $\xi\ra h_m(\xi+ai,\eta), \ \xi\in\mathbb{T}^d$, have the following properties: \\
Let $\rho=[\rho_1,\dots,\rho_d]$ be a $d-$tuple of non-negative integers with norm $|\rho|=\rho_1+\cdots+\rho_d$.  If $|\rho|<d$ the derivative
\be \label{BH3}
\prod_{j=1}^d \left(\frac{\pa}{\pa \xi_j}\right)^{\rho_j} h_m(\xi+ai,\eta)
\ee
is in $L^p_w(\mathbb{T}^d;\mathcal{L}(\mathbb{C}^d))$ with $p=d/|\rho|$, and its norm is bounded as
\be \label{BI3}
\left\|                 \prod_{j=1}^d \left(\frac{\pa}{\pa \xi_j}\right)^{\rho_j} h_m(\xi+ai,\eta)           \right\|_{p,w} \ \le \ Mm^{3|\rho|}(1-\la)^{m+1}\left[     1+\frac{C_d|a|^2}{\eta}     \right]^m \ , 
\ee
where the constant $M$ depends only on $d$. If $|\rho|=d-1$ then the fractional derivative 
\be \label{BW3}
\frac{1}{|\De\xi|^{1-\nu}}\prod_{j=1}^d \left(\frac{\pa}{\pa \xi_j}\right)^{\rho_j}\left[ h_m(\xi+\De\xi+ai,\eta)-h_m(\xi+ai,\eta)\right]
\ee
with $0<\nu\le 1$ and  $|\De\xi|\le 1$ is in $L^p_w(\mathbb{T}^d;\mathcal{L}(\mathbb{C}^d))$ with $p=d/(d-\nu)$, and its norm is bounded by the RHS of (\ref{BI3}) with  $m^{3d}$ in place of $m^{3|\rho|}$. The constant
$M$ depends now on $\nu>0$ as well as $d$. 
\end{proposition}
\begin{proof}
We see from (\ref{A3}), (\ref{D3}), (\ref{M3}) that
\begin{multline} \label{BJ3}
\frac{\pa}{\pa \xi_j} h_m(\xi+ai)  \ = \  \sum_{r=1}^m \left\langle \  \mathbf{b}(\cdot)[PT_{\xi+ai,\eta}\mathbf{b}(\cdot)]^{r-1} T_{m-r,\eta}f(\xi+ai) \  \right\rangle  \  , \\
{\rm where \ } f(\zeta) \ = \ \frac{\pa}{\pa \zeta_j}\left\{ \frac{e(\zeta)e^*(\zeta)}{\eta+e(\zeta)^*e(\zeta)}\right\} \ .
\end{multline}
The function $\xi\ra f(\xi+ai), \ \xi\in\mathbb{T}^d,$ is in $L^d_w(\mathbb{T}^d;\mathcal{L}(\mathbb{C}^d))$ with norm bounded by a constant depending only on $d$. Arguing as in Lemma 3.7, we have 
from the Schwarz inequality that for $v_1,v_2\in\mathbb{C}^d$,
\begin{multline} \label{BK3}
\left|v_1^*\left\langle \  \mathbf{b}(\cdot)[PT_{\xi+ai,\eta}\mathbf{b}(\cdot)]^{r-1} T_{m-r,\eta}f(\xi+ai) \  \right\rangle v_2  \right|  \\
 = \ \left|\left[   \mathbf{b}(\cdot) \{PT_{\xi-ai,\eta}\mathbf{b}(\cdot)\}^{r-1} v_1,    T_{m-r,\eta}f(\xi+ai) v_2\  \right]_{L^2(\Om;\mathbb{C}^d)}\right| \\
 \le \ (1-\la)^{r}\left[     1+\frac{C_d|a|^2}{\eta}     \right]^{r-1}|v_1| \left\|   T_{m-r,\eta}f(\xi+ai) v_2                        \right\|_{L^2(\Om;\mathbb{C}^d)} \ .
\end{multline}
Taking the supremum in (\ref{BK3}) over $\{v_1\in\mathbb{C}^d: \ |v_1|=1\}$ and summing  through an orthonormal basis of $v_2\in\mathbb{C}^d$, we see that (\ref{BK3}) implies the inequality
\begin{multline} \label{BL3}
\left\|\left\langle \  \mathbf{b}(\cdot)[PT_{\xi+ai,\eta}\mathbf{b}(\cdot)]^{r-1} T_{m-r,\eta}f(\xi+ai) \  \right\rangle  \right\|_{\rm HS}   \\
 \le \ (1-\la)^{r}\left[     1+\frac{C_d|a|^2}{\eta}     \right]^{r-1}\left\|   T_{m-r,\eta}f(\xi+ai)                       \right\|_{\rm HS} \ .
\end{multline}
It follows now from Lemma 3.8 and (\ref{BL3}) that
\begin{multline} \label{BM3}
\left\|\left\langle \  \mathbf{b}(\cdot)[PT_{\cdot+ai,\eta}\mathbf{b}(\cdot)]^{r-1} T_{m-r,\eta}f(\cdot+ai) \  \right\rangle  \right\|_{d,w}   \\
 \le \ M\sqrt{d}(1-\la)^{m+1}\left[     1+\frac{C_d|a|^2}{\eta}     \right]^{m-1}(m+2-r)\left\|  f(\cdot+ai)                       \right\|_{d,w} \ .
\end{multline}
Next we use the quasi sub-additivity of weak norms.  Let $(\mathcal{M},\mu)$ be a measure space and $g_k:\mathcal{M}\ra\mathbb{R}, \ k=1,\dots,N,$ be measurable functions such that each $g_k\in L_w^p(\mathcal{M};\mathbb{R})$. Then $g=g_1+\cdots+ g_N$ is also in $L_w^p(\mathcal{M};\mathbb{R})$  and
\be \label{BN3}
\|g\|_{p,w} \ \le \ N\left\{ \sum_{j=1}^N \|g_j\|_{p,w}^p\right\}^{1/p} \ .
\ee
We see from  (\ref{BJ3}), (\ref{BM3}), (\ref{BN3})  that (\ref{BI3}) holds when $d>2,|\rho|=1$.

More generally, for $|\rho|\ge 1$, the derivative (\ref{BH3}) is a sum of $m^{|\rho|}$ terms. Each term has the form
\begin{multline} \label{BO3}
\left\langle \  \mathbf{b}(\cdot)[PT_{\xi+ai,\eta}\mathbf{b}(\cdot)]^{r_0}   \left\{\prod_{j=1}^n  T_{r_j,\eta} f_j\right\}(\xi+ai)    \  \right\rangle  \  , \quad
{\rm where \ } r_0+r_1+\cdots+r_n \ = \ m-n , \\  f_j(\cdot+ai)\in L^{p_j}_w(\mathbb{T}^d;\mathcal{L}(\mathbb{C}^d)), \  j=1,\dots, n,  \quad 
{\rm and \ } \sum_{j=1}^n \frac{1}{p_j} \ = \ \frac{|\rho|}{d} \ , \quad n\le |\rho|  \ .
\end{multline}
Proceeding as in (\ref{BL3}), (\ref{BM3}) and using Lemma 3.8 again, we see the expression in (\ref{BO3}) is bounded as
\begin{multline} \label{BP3}
\left\|\left\langle \  \mathbf{b}(\cdot)[PT_{\cdot+ai,\eta}\mathbf{b}(\cdot)]^{r_0}   \left\{\prod_{j=1}^n  T_{r_j,\eta} f_j\right\}(\cdot+ai)    \  \right\rangle \right\|_{d/|\rho|,w}  \\
\le M(1-\la)^{m+1}\left[     1+\frac{C_d|a|^2}{\eta}     \right]^{m-n} \prod_{j=1}^n (r_j+2)\|f_j(\cdot+ai)\|_{p_j,w} \ \ ,
\end{multline}
provided $|\rho|<d/2$.  Note that the constant $M$ in (\ref{BO3}) depends only on $d$. Now summing (\ref{BP3}) over all $m^{|\rho|}$ terms for the derivative (\ref{BH3}) and using (\ref{BN3}), we see that (\ref{BI3}) holds provided $|\rho|<d/2$. 

In the case when $d/2\le |\rho|<d$ we use an identity as in (\ref{BK3}). Thus we have  for the expression  (\ref{BO3}) the identity,
\begin{multline} \label{BQ3}
v_1^*\left\langle \  \mathbf{b}(\cdot)[PT_{\xi+ai,\eta}\mathbf{b}(\cdot)]^{r_0}   \left\{\prod_{j=1}^n  T_{r_j,\eta} f_j\right\}(\xi+ai)    \  \right\rangle v_2 \\
= \ \left[   \mathbf{b}(\cdot) \{PT_{\xi-ai,\eta}\mathbf{b}(\cdot)\}^{r_k} \left\{\prod_{j=k}^1  T_{r_{j-1},\eta} f_j\right\}(\xi-ai)   v_1,     \left\{\prod_{j=k+1}^n  T_{r_j,\eta} f_j\right\}(\xi+ai)   v_2\  \right]_{L^2(\Om;\mathbb{C}^d)} \ .
\end{multline}
We choose $k$ if possible such that
\be \label{BR3}
\sum_{j=1}^k \frac{1}{p_j} =\frac{1}{p} < \frac{1}{2} \ , \quad  \sum_{j=k+1}^n \frac{1}{p_j} =\frac{1}{q} < \frac{1}{2}  , \quad {\rm noting \ that \ } \ \quad \frac{1}{p}+\frac{1}{q} \ = \ \sum_{j=1}^n \frac{1}{p_j} \ = \ \frac{|\rho|}{d}<1 \ .
\ee
Applying the Schwarz inequality to the RHS of (\ref{BQ3}),  then taking the supremum over $\{v_1\in\mathbb{C}^d: \ |v_1|=1\}$ and summing  through an orthonormal basis of $v_2\in\mathbb{C}^d$, we obtain the inequality,
\begin{multline} \label{BS3}
\left\|\left\langle \  \mathbf{b}(\cdot)[PT_{\xi+ai,\eta}\mathbf{b}(\cdot)]^{r_0}   \left\{\prod_{j=1}^n  T_{r_j,\eta} f_j\right\}(\xi+ai)    \  \right\rangle\right\|_{\rm HS}  \\
\le \ (1-\la)^{r_k+1}\left[     1+\frac{C_d|a|^2}{\eta}     \right]^{r_k}\left\|\left\{\prod_{j=k}^1  T_{r_{j-1},\eta} f_j\right\}(\xi-ai)\right\|_{\rm HS}  \left\| \left\{\prod_{j=k+1}^n  T_{r_j,\eta} f_j\right\}(\xi+ai) \right\|_{\rm HS} \ .
\end{multline}
Let $f\in L_w^p(\mathcal{M};\mathbb{R})$ and $g\in L_w^q(\mathcal{M};\mathbb{R})$, where $1/p+1/q=1/r\le1$. Then
the function $fg$ is in  $L_w^r(\mathcal{M};\mathbb{R})$  and satisfies the inequality  $\|fg\|_{r,w}\le 2^{1/r}\|f\|_{p,w}\|g\|_{q,w}$.  It follows then from (\ref{BS3}) and Lemma 3.8 that the inequality (\ref{BP3}) continues to hold (with slight modification) when one can find $k$ satisfying (\ref{BR3}).  In the case when one cannot find $k$ satisfying (\ref{BR3}) there still exists $k$ such that
\be \label{BT3}
\sum_{j=1}^k \frac{1}{p_j}  < \frac{1}{2} \ , \quad  \sum_{j=1}^{k+1} \frac{1}{p_j} \ge \frac{1}{2} \ , \quad  \sum_{j=1}^n \frac{1}{p_j} \ = \ \frac{|\rho|}{d}<1 \ .
\ee
In the case of (\ref{BT3}) we factorize $f_{k+1}$ so $f_{k+1}(\zeta)=f^a_{k+1}(\zeta)f^b_{k+1}(\zeta), \ \zeta\in\mathbb{R}^d,$ in such a way that $f^a_{k+1}(\cdot-ai)$ is in $L^{p^a_{k+1}}_w(\mathbb{T}^d;\mathcal{L}(\mathbb{C}^d))$ 
and  $f^b_{k+1}(\cdot+ai)$ is in $L^{p^b_{k+1}}_w$,
where  
\be \label{BU3}
\sum_{j=1}^k \frac{1}{p_j} +\frac{1}{p^a_{k+1}}=\frac{1}{p} < \frac{1}{2} \ , \quad  \frac{1}{p^b_{k+1}} +\sum^{n}_{j=k+2} \frac{1}{p_j}=\frac{1}{q} < \frac{1}{2} \ , \quad \frac{1}{p}+\frac{1}{q}=\frac{|\rho|}{d} \ .
\ee
The LHS of (\ref{BQ3}) is equal to
\begin{multline} \label{BV3}
 \Bigg[   T_{r_k,\eta} f^a_{k+1} \left\{\prod_{j=k}^1  T_{r_{j-1},\eta} f_j\right\}(\xi-ai)   v_1,  \\
  T_{r_{k+1},\eta} f^b_{k+1}   \left\{\prod_{j=k+2}^n  T_{r_j,\eta} f_j\right\}(\xi+ai)   v_2\  \Bigg]_{L^2(\Om;\mathbb{C}^d)} \ .
\end{multline}
In view of (\ref{BU3}) we may apply the Schwarz inequality and Lemma 3.8 to (\ref{BV3}) to obtain an inequality like (\ref{BP3}). 

We proceed similarly to prove (\ref{BW3}).  We write
\begin{multline} \label{BX3}
|\De\xi|^{-(1-\nu)}\left[h_m(\xi+\De\xi+ai,\eta)-h_m(\xi+ai,\eta)\right] \\
 = \ \sum_{r=1}^m \left\langle \  \mathbf{b}(\cdot)[PT_{\xi+ai,\eta}\mathbf{b}(\cdot)]^{r-1} T_{m-r,\eta}f_{\nu,\De\xi}(\xi+\De\xi+ai) \  \right\rangle   \ , \\
{\rm where \ } f_{\nu,\De\xi}(\zeta) \ = \ \frac{f(\zeta)-f(\zeta-\De\xi)}{|\De\xi|^{1-\nu}} \ , \quad f(\zeta) \ = \ \frac{e(\zeta)e^*(\zeta)}{\eta+e(\zeta)^*e(\zeta)} \ .
\end{multline}
We then observe that for any  $\sig=[\sig_1,\dots,\sig_d]$, a  $d-$tuple of non-negative integers satisfying $|\sig|<d$, the derivative
\be \label{BY3}
  \prod_{j=1}^d \left(\frac{\pa}{\pa \zeta_j}\right)^{\sig_j} f_{\nu,\De\xi}(\zeta+ai) \ , \quad \zeta\in\mathbb{T}^d,
\ee 
is in $L^p_w(\mathbb{T}^d;\mathcal{L}(\mathbb{C}^d))$ with  $p=d/[|\sig|+1-\nu]$.  The norm of the function (\ref{BY3}) depends only  on $d$, and not on $|\De\xi|\le 1$ or  $\nu, \ 0\le\nu\le1$.  The argument then goes as in the previous paragraph,
noting only that the constant $M$ in (\ref{BW3}) diverges as $\nu\ra0$ since $d/(d-\nu)$ approaches $1$. 
\end{proof}
\begin{corollary}
Let $q(\cdot,\cdot)$ be the function defined by  (\ref{D3}).  There exists a constant $c_d$, depending only on $d$, such that for fixed $\eta$ with $0<\eta\le 1$, and $a\in\mathbb{R}^d$ satisfying $|a|^2<c_d\la \eta$  the derivatives of the function  $\xi\ra q(\xi+ai,\eta), \ \xi\in\mathbb{T}^d$, have the following properties: \\
If $|\rho|<d$ the derivative
\be \label{BZ3}
\prod_{j=1}^d \left(\frac{\pa}{\pa \xi_j}\right)^{\rho_j} q(\xi+ai,\eta)
\ee
is in $L^p_w(\mathbb{T}^d;\mathcal{L}(\mathbb{C}^d))$ with $p=d/|\rho|$, and its norm is bounded as
\be \label{CA3}
\left\|                 \prod_{j=1}^d \left(\frac{\pa}{\pa \xi_j}\right)^{\rho_j} q(\xi+ai,\eta)           \right\|_{p,w} \ \le \ \frac{M}{\la^{3|\rho|+2}}  \ , 
\ee
where the constant $M$ depends only on $d$. If $|\rho|=d-1$ then the fractional derivative 
\be \label{CB3}
\frac{1}{|\De\xi|^{1-\nu}}\prod_{j=1}^d \left(\frac{\pa}{\pa \xi_j}\right)^{\rho_j}\left[ q(\xi+\De\xi+ai,\eta)-q(\xi+ai,\eta)\right]
\ee
with $0<\nu\le 1$ and  $|\De\xi|\le 1$ is in $L^p_w(\mathbb{T}^d;\mathcal{L}(\mathbb{C}^d))$ with $p=d/(d-\nu)$, and its norm is bounded by the RHS of (\ref{CA3}) with  $\la^{3d+2}$ in place of $\la^{3|\rho|+2}$. The constant
$M$ depends now on $\nu>0$ as well as $d$. 
\end{corollary}
\begin{proof}
The result follows from Proposition 3.1 We need to  use the quasi sub-additivity of weak norms in a slightly more general  form than in (\ref{BN3}).  Let  $g_k:\mathcal{M}\ra\mathbb{R}, \ k=1,\dots,N,$ be measurable functions such that each $g_k\in L_w^p(\mathcal{M};\mathbb{R})$. Then  the function  $g=g_1+ g_2\cdots+ g_N$ is also in $L_w^p(\mathcal{M};\mathbb{R})$  and
\begin{multline} \label{CD3}
\|g\|_{p,w} \ \le \ \left\{ \sum_{m=1}^N \frac{\|g_m\|_{p,w}^p}{\beta_m^p}\right\}^{1/p} \quad  {\rm for \ any \ positive  \ sequence \ } \{\beta_m, \ m\ge 1\} \\
{\rm which  \ satisfies \ } \sum_{m=1}^N\beta_m \ = \ 1 \ .
\end{multline}
To prove (\ref{CA3}) we take $g_m$ as in (\ref{BH3}). From (\ref{BI3}) we have that $\|g_m\|_{p,w}\le Mm^{3|\rho|}(1-\la/2)^m$  provided $|a|^2\le \la\eta/2C_d$.  We choose now $\beta_m=\la(1-\la/4)^{m-1}/4$, whence
\be \label{CE3}
\left\{\sum_{m=1}^N \frac{\|g_m\|_{p,w}^p}{\beta_m^p}\right\}^{1/p} \ \le\ \frac{4M}{\la}\left\{ \sum_{m=1}^N  m^{3|\rho|p}(1-\la/4)^{mp} \right\}^{1/p} \ \le \ \frac{M_p}{\la^{3|\rho|+1+1/p}} \ ,
\ee
for some constant $M_p$ depending on $M,p$.  The inequality (\ref{CE3}) establishes (\ref{CA3}), and a similar argument the bound on (\ref{CB3}). 
\end{proof}
\begin{lem}
Let $q(\cdot,\cdot)$ be the function defined by  (\ref{D3}).  There exists a constant $c_d$, depending only on $d$, such that for fixed $\eta$ with $0<\eta\le 1$, and $a\in\mathbb{R}^d$ satisfying $|a|^2<c_d\la \eta$, the function  $q(\cdot,\cdot)$ satisfies the quadratic form inequalities,
\be \label{CF3}
\la|v|^2 \ \le \ \Re\left[v^*q(\xi+ai,\eta)v\right] \ \le \ \ |v|^2 \ ,  \quad \xi\in\mathbb{R}^d  ,  \ v\in\mathbb{C}^d  ,
\ee
\be \label{CG3}
\left|\Im\left[v^*q(\xi+ai,\eta)v\right]\right| \ \le \ |v|^2 \ ,  \quad \xi\in\mathbb{R}^d   ,  \ v\in\mathbb{C}^d  .
\ee
\end{lem}
\begin{proof}
We have from (2.5), (2.9) of \cite{cs1} that
\begin{multline} \label{CH3}
q(\xi,\eta) \ = \  \left\langle \mathbf{a}(\cdot)\right\rangle+ \left\langle \mathbf{a}(\cdot)\Psi(\xi,\eta,\cdot)\right\rangle \ ,  \\
{\rm where \ } \Psi(\xi,\eta,\cdot) \ = \  PT_{\xi,\eta}\left[\mathbf{b}(\cdot)\Psi(\xi,\eta,\cdot)\right]+ PT_{\xi,\eta}\mathbf{b}(\cdot) \ .
\end{multline}
 It follows from (\ref{CH3}) that
\begin{multline} \label{CI3}
\left[ \Psi(\xi,\eta,\cdot)v, \  \Psi(\xi,\eta,\cdot)v\right]_{L^2(\Om;\mathbb{C}^d)} \ = \ 
\left[  T^*_{\xi,\eta} \Psi(\xi,\eta,\cdot)v,  \ \mathbf{b}(\cdot)\Psi(\xi,\eta,\cdot)v      \right]_{L^2(\Om;\mathbb{C}^d)} \\
-  \left[  T^*_{\xi,\eta} \Psi(\xi,\eta,\cdot)v, \ \left\{\mathbf{a}(\cdot)-\la I_d\right\}v     \right]_{L^2(\Om;\mathbb{C}^d)}  \ ,
\end{multline}
where $T^*_{\xi,\eta}$ is the adjoint of the operator $T_{\xi,\eta}$ on  $L^2(\Om;\mathbb{C}^d)$.  In (\ref{CI3}) we are using the fact that $P\Psi(\xi,\eta,\cdot)=\Psi(\xi,\eta,\cdot)$, whence $\langle \ \Psi(\xi,\eta,\cdot) \ \rangle=0$. 

We apply the Schwarz inequality to the RHS of (\ref{CI3}), using the fact that the matrices $\mathbf{b}(\om)$ and $\mathbf{a}(\om)-\la I_d$ are symmetric positive definite.  Thus we obtain the inequality
\begin{multline} \label{CJ3}
\left[ \Psi(\xi,\eta,\cdot)v, \  \Psi(\xi,\eta,\cdot)v\right]_{L^2(\Om;\mathbb{C}^d)}   \\
 \le \ \frac{1}{2}  \left[  T^*_{\xi,\eta} \Psi(\xi,\eta,\cdot)v,  \ \mathbf{b}(\cdot) T^*_{\xi,\eta} \Psi(\xi,\eta,\cdot)v      \right]_{L^2(\Om;\mathbb{C}^d)}  +
 \frac{1}{2}  \left[  \Psi(\xi,\eta,\cdot)v,  \ \mathbf{b}(\cdot)  \Psi(\xi,\eta,\cdot)v      \right]_{L^2(\Om;\mathbb{C}^d)}  \\
 + \ \frac{1}{2}  \left[  T^*_{\xi,\eta} \Psi(\xi,\eta,\cdot)v,  \ \left\{\mathbf{a}(\cdot)-\la I_d\right\} T^*_{\xi,\eta} \Psi(\xi,\eta,\cdot)v      \right]_{L^2(\Om;\mathbb{C}^d)} 
 +\frac{1}{2}v^*\left\langle \left\{\mathbf{a}(\cdot)-\la I_d\right\}\right\rangle v \\
 = \    \  \frac{1-\la}{2}  \left[  T^*_{\xi,\eta} \Psi(\xi,\eta,\cdot)v,  \  T^*_{\xi,\eta} \Psi(\xi,\eta,\cdot)v      \right]_{L^2(\Om;\mathbb{C}^d)}  \\
 + \   \frac{1}{2}  \left[  \Psi(\xi,\eta,\cdot)v,  \ \mathbf{b}(\cdot)  \Psi(\xi,\eta,\cdot)v      \right]_{L^2(\Om;\mathbb{C}^d)} +\frac{1}{2}v^*\left\langle \left\{\mathbf{a}(\cdot)-\la I_d\right\}\right\rangle v \ .
\end{multline}
We have from Lemma 3.1 that  $\|T_{\xi+ai,\eta}\|_{L^2(\Om;\mathbb{C}^d)}\le 1+c_dC_d\la$, whence  on choosing $c_d<1/3C_d$ and replacing $\xi$ by $\xi+ai$ in (\ref{CJ3}), we obtain  the inequality,
\begin{multline} \label{CK3}
\left[ \Psi(\xi+ai,\eta,\cdot)v, \  \Psi(\xi+ai,\eta,\cdot)v\right]_{L^2(\Om;\mathbb{C}^d)}  \\
 \le \    \left[  \Psi(\xi+ai,\eta,\cdot)v,  \ \{I_d-\mathbf{a}(\cdot) \} \Psi(\xi+ai,\eta,\cdot)v      \right]_{L^2(\Om;\mathbb{C}^d)} +v^*\left\langle \left\{\mathbf{a}(\cdot)-\la I_d\right\}\right\rangle v \ .
\end{multline}
It follows from (\ref{CK3}) that
\be \label{CL3}
\left[ \Psi(\xi+ai,\eta,\cdot)v, \  \mathbf{a}(\cdot)\Psi(\xi+ai,\eta,\cdot)v\right]_{L^2(\Om;\mathbb{C}^d)}   \ \le \ v^*\left\langle \left\{\mathbf{a}(\cdot)-\la I_d\right\}\right\rangle v \ .
\ee
We then have from (\ref{CH3}),  (\ref{CL3}) that 
\be \label{CM3}
|v^*q(\xi+ai,\eta)v-v^*\langle\mathbf{a}(\cdot)\rangle v | \ = \  |v^*\langle\{\mathbf{a}(\cdot)-\la I_d\}\Psi(\xi+ai,\eta,,\cdot)\rangle v| \ \le \  v^*\left\langle \left[\mathbf{a}(\cdot)-\la I_d\right]\right\rangle v \ .
\ee
We see from (\ref{CM3}) that the lower bound in (\ref{CF3}) and also (\ref{CG3}) hold.

To prove the upper bound in (\ref{CF3}) we observe from (\ref{CH3}) and the Schwarz inequality  that
\begin{multline} \label{CN3}
\left[ \Psi(\xi,\eta,\cdot)v, \  \Psi(\xi,\eta,\cdot)v\right]_{L^2(\Om;\mathbb{C}^d)} \ = \ 
\left[  T^*_{\xi,\eta} \Psi(\xi,\eta,\cdot)v,  \ \mathbf{b}(\cdot)\{\Psi(\xi,\eta,\cdot)+I_d\}v      \right]_{L^2(\Om;\mathbb{C}^d)}  \\
\le \ \frac{1}{2} \left[  T^*_{\xi,\eta} \Psi(\xi,\eta,\cdot)v,  \ \mathbf{b}(\cdot) T^*_{\xi,\eta}  \Psi(\xi,\eta,\cdot)v      \right]_{L^2(\Om;\mathbb{C}^d)} +
\  \frac{1}{2}\left[  \{\Psi(\xi,\eta,\cdot)+I_d\}v,  \ \mathbf{b}(\cdot)\{\Psi(\xi,\eta,\cdot)+I_d\}v      \right]_{L^2(\Om;\mathbb{C}^d)}  \\
\le \frac{1-\la}{2} \left[  T^*_{\xi,\eta} \Psi(\xi,\eta,\cdot)v,  \ T^*_{\xi,\eta}  \Psi(\xi,\eta,\cdot)v      \right]_{L^2(\Om;\mathbb{C}^d)} +
\  \frac{1}{2}\left[  \{\Psi(\xi,\eta,\cdot)+I_d\}v,  \ \mathbf{b}(\cdot)\{\Psi(\xi,\eta,\cdot)+I_d\}v      \right]_{L^2(\Om;\mathbb{C}^d)}  \ .
\end{multline}
Replacing $\xi$ in (\ref{CN3}) by $\xi+ai$ and  again choosing $c_d<1/3C_d$ we conclude that
\be \label{CO3}
\Re\left[v^*\left\langle \mathbf{b}(\cdot)\Psi(\xi+ai, \eta,,\cdot)     \right\rangle v \right]+\frac{1}{2} v^*\left\langle \mathbf{b}(\cdot)    \right\rangle v  \ \ge \ 0 \ .
\ee
Now (\ref{CO3}) implies that
\be \label{CP3}
\Re\left[v^*q(\xi+ai,\eta)v\right] \ = \ v^*\left\langle \mathbf{a}(\cdot)    \right\rangle v-\Re\left[v^*\left\langle \mathbf{b}(\cdot)\Psi(\xi+ai,\eta,,\cdot)     \right\rangle v \right] \ \le \  \frac{1}{2} v^*\left\langle  I_d+\mathbf{a}(\cdot)    \right\rangle v \ ,
\ee
whence the upper bound  in (\ref{CF3}) follows. 
\end{proof}
Lemma  3.9 implies that the function $\xi\ra q(\xi+ai,\eta)$ is bounded independent of $\eta, \ 0<\eta\le 1$, and  Corollary 3.1 indicates that the function has weak differentiability properties.  We state a hypothesis analogous to Hypothesis 3.1
of \cite{cs1}, which was established for various environments  in that paper:
\begin{hyp}
For some $\al>0$ and constant $C_\al$  the function $q(\cdot,\cdot)$  satisfies an inequality
\be \label{CQ3}
\| q(\xi,\eta)-q(\xi',\eta') \|_{\mathcal{L}(\mathbb{C}^d)} \ \le \  C_\al\left[ |\xi-\xi'|^\al+|\eta-\eta'|^{\al/2}\right] \  , \quad \xi,\xi'\in\mathbb{C}^d, \ |\Im\xi|,|\Im\xi'|\le c_d\sqrt{\la\eta} \ ,
\ee
where $c_d$ depends only on $d$. 
\end{hyp}

\vspace{.1in}
\section{Configuration space estimates from Fourier space estimates}
It was already shown in \cite{cn1,cn2} that the averaged Green's function $G_{\mathbf{a},\eta}(x), \ x\in\mathbb{Z}^d$,  has the Fourier inverse representation,
\be \label{A4}
G_{\mathbf{a},\eta}(x)  \ = \  \frac{1}{(2\pi)^d}\int_{\mathbb{T}^d} \frac{e^{-ix\cdot\xi} }{\eta+e(\xi)^*q(\xi,\eta)e(\xi)} \ d\xi \ , \quad x\in\mathbb{Z}^d \ ,
\ee
where $q(\cdot,\cdot)$ is the function (\ref{D3}).  We see from (\ref{A4}) that the second derivative of $G_{\mathbf{a},\eta}(\cdot)$ is   given by the formula,
\be \label{B4}
\na\na^*G_{\mathbf{a},\eta}(x)  \ = \  \frac{1}{(2\pi)^d}\int_{\mathbb{T}^d} \frac{e(\xi)e(\xi)^*e^{-ix\cdot\xi} }{\eta+e(\xi)^*q(\xi,\eta)e(\xi)} \ d\xi \ , \quad x\in\mathbb{Z}^d \ .
\ee
One may then conclude  from the representation (\ref{B4}), Corollary 3.1 and Lemma 3.9 that $\na\na^*G_{\mathbf{a},\eta}(\cdot)$ satisfies an inequality
\be \label{C4}
\left|            \na\na^*G_{\mathbf{a},\eta}(x)         \right| \ \le \   \frac{Ce^{-\ga\sqrt{\eta}|x|}}{(1+|x|)^{d-\nu}} \ , \quad x\in\mathbb{Z}^d \ ,  \ 0<\eta\le 1, \ \nu>0 \ ,
\ee
for constants $C,\ga$ depending only on $d,\la,\nu$.  Note that the constants in (\ref{C4})  do not depend on $\eta$, so we may let $\eta\ra0$. 

The inequality (\ref{C4}) was established in \cite{cn2, cs1}. Our first goal here is to generalize this.  We consider a  function $F(\xi,\eta)$ which has the form
\be \label{D4}
F(\xi,\eta) \ = \ \frac{P_k(q(\xi,\eta), e(\xi),e(\xi)^*)}{[\eta+e(\xi)^*q(\xi,\eta)e(\xi)]^k}  \ ,
\ee
where $k$ is a positive integer and $P_k$ is a homogeneous polynomial of degree $2k$ in the variables $e_j(\xi),e^*_j(\xi), \ j=1,\dots,d$ with coefficients being products of entries of the $d\times d$ matrix $q(\xi,\eta)$.. 
\begin{lem}
The Fourier inverse $x\ra\hat{F}(x,\eta), \ x\in\mathbb{Z}^d,$ of the function (\ref{D4}) satisfies an inequality,
\be \label{E4}
| \hat{F}(x,\eta) | \ \le \   \frac{C_{d,\nu}e^{-c_d\sqrt{\la\eta}|x|}}{(1+|x|)^{d-\nu}} \ , \quad x\in\mathbb{Z}^d \ ,  \ 0<\eta\le 1, \ \nu>0 \ ,
\ee
for constants $C_{d,\nu}$ depending only  on $d,\nu$ and $c_d>0$ on $d$.  
\end{lem}
\begin{proof}  We may obtain from Lemma 3.9  upper and lower bounds on the denominator in (\ref{D4}) with $\xi$ replaced by $\xi+ai$. The analytic continuation of the function $\xi\ra e(\xi)^*q(\xi,\eta)e(\xi), \ \xi\in\mathbb{R}^d,$
to its translation by $ai\in\mathbb{C}^d$ is  the function $\xi\ra e(\xi-ai)^*q(\xi+ai,\eta)e(\xi+ai), \ \xi\in\mathbb{R}^d$. We write
\begin{multline} \label{F4}
 e(\xi-ai)^*q(\xi+ai,\eta)e(\xi+ai) \ = \  e(\xi)^*q(\xi+ai,\eta)e(\xi) \\
 +\{e(\xi-ai)-e(\xi)\}^*q(\xi+ai,\eta)e(\xi) +e(\xi)^*q(\xi+ai,\eta)\{e(\xi+ai)-e(\xi)\} \\
 + \{e(\xi-ai)-e(\xi)\}^*q(\xi+ai,\eta)\{e(\xi+ai)-e(\xi)\} \ . 
\end{multline}
We consider  for $A\in\mathcal{L}(\mathbb{C}^d)$   the self-adjoint matrices $\Re A=(A+A^*)/2, \Im A=(A-A^*)/2i$ so $A=\Re A+i\Im A$.  We have from  Lemma 3.9 that  if $A=q(\xi+ai,\eta)$, then 
\be \label{G4}
\la I_d \ \le \ \Re A \ \le \ I_d \ , \quad -I_d \ \le \  \Im A \ \le I_d \ ,
\ee
whence it follows that $\|A\|\le 2$.  We conclude from (\ref{G4}) that
\be \label{H4}
\left|\{e(\xi-ai)-e(\xi)\}^*q(\xi+ai,\eta)e(\xi) \right| \ \le \  4|a||e(\xi)| \  \le \ \frac{16|a|^2}{\la} + \frac{\la |e(\xi)|^2}{4} \ , \quad \xi\in \mathbb{R}^d \ ,
\ee
provided $|a|<1$, with a similar inequality for the third term on the RHS of (\ref{F4}).  We conclude from (\ref{F4})-(\ref{H4}) that
\begin{multline} \label{I4}
\frac{1}{2}[\eta+ \la|e(\xi)|^2] \ \le \ \eta+\Re\{ e(\xi-ai)^*q(\xi+ai,\eta)e(\xi+ai) \} \ \le \ \{\eta+|e(\xi)|^2\}+\frac{1}{2}[\eta+ \la|e(\xi)|^2]  \  , \\
\left|    \Im\{ e(\xi-ai)^*q(\xi+ai,\eta)e(\xi+ai) \} \                           \right| \ \le \    |e(\xi)|^2+ \frac{1}{2}[\eta+ \la|e(\xi)|^2]  \  , 
\end{multline}
provided the constant $c_d$ in Lemma 3.9   is sufficiently small.
   
Analytically continuing the function $\xi\ra F(\xi,\eta), \ \xi\in\mathbb{R}^d$,  to $\mathbb{C}^d$, we obtain the representation
\be \label{J4}
\hat{F}(x,\eta) \ = \ \frac{e^{a\cdot x}}{(2\pi)^d}\int_{\mathbb{T}^d}  F(\xi+ai,\eta) e^{-i\xi\cdot x} \  d\xi \ ,   \quad x\in\mathbb{Z}^d  \ .
\ee
The analytic continuation of the numerator  in (\ref{D4}) by a translation $ai$ is bounded as
\be \label{K4}
\left|   P_k(q(\xi+ai,\eta),e(\xi+ai),e(\xi-ai)^*)                      \right| \ \le \  C\{\eta+|e(\xi)|^2\}^k \ , \quad \xi\in\mathbb{R}^d \ ,
\ee
for some constant $C$.  Now (\ref{I4}), (\ref{K4}) imply that $\sup_{\xi\in\mathbb{T}^d}|F(\xi+ai, \eta)|$ is bounded, whence the integral in (\ref{J4}) is bounded by a constant independent of $\eta, \ 0<\eta\le 1$. 
On choosing $a$ in (\ref{J4}) appropriately, with $a$ satisfying the condition of Lemma 3.9, we obtain the exponential falloff in (\ref{E4}).

To obtain the polynomial decay in $(1+|x|)$ we assume wlog that $|x_1|\ge |x|/d$.
Using integration by parts we write the integral in (\ref{J4}) as
\be \label{M4}
I_{d-1}(x) \ = \ \frac{1}{(ix_1)^{d-1}}\int_{\mathbb{T}^d} e^{-i\xi\cdot x} \left(\frac{\pa}{\pa\xi_1}\right)^{d-1}F(\xi+ai,\eta) \  d\xi  \ . 
\ee
We have that
\be \label{N4}
\left(\frac{\pa}{\pa\xi_1}\right)^{d-1}F(\xi+ai,\eta)  \ = \ P\left(  \xi+ai,\eta,\frac{\pa q(\xi+ai,\eta)}{\pa\xi_1}, \dots,      \frac{\pa^{d-1} q(\xi+ai,\eta)}{\pa\xi_1^{d-1}}                           \right) \ ,
\ee
where $P$ is a polynomial in $\pa q(\xi,\eta)/\pa\xi_1,\dots,\pa ^{d-1}q(\xi,\eta)/\pa\xi_1^{d-1}$, with coefficients depending on $\xi+ai,\eta, q(\xi+ai,\eta)$.  A typical term in the polynomial has the form
\be \label{O4}
A(\xi+ai,\eta)\prod_{j=1}^{d-1} \frac{\pa ^{r_j}q_{m_j,n_j}(\xi+ai,\eta)}{\pa\xi_1^{r_j}} \ , \quad {\rm with \ } r_1+\cdots+r_{d-1}=r \ \le  \ d-1 \ ,
\ee
where the $r_j$ are non-negative integers.  One also sees that
\be \label{P4}
|A(\xi+ai,\eta)| \ \le \ \frac{C}{|\xi|^{d-1-r}} \ , \quad \xi\in\mathbb{T}^d \ .
\ee

To bound the integral on  $\mathbb{T}^d$ 
of the term (\ref{O4}) satisfying (\ref{P4}), we use the fact that if  $f_j\in L^{p_j}_w(\mathcal{M};\mathbb{R}), \ j=1,\dots,N,$ with 
 $1/p_1+\cdots+1/p_N=1/p\le1$, then $\prod_{j=1}^N f_j$ is in $L^p_w(\mathcal{M};\mathbb{R})$ and   $\left\|\prod_{j=1}^N f_j\right\|_{p,w}\le N^{1/p}\prod_{j=1}^N\|f_j\|_{p_j,w}$.  Then we apply a classical
 inequality for functions in weak spaces: Let  $g\in L^p_w(\mathcal{M};\mathbb{R})$ with $p>1$ and $E\subset \mathcal{M}$ a measurable set. 
 Then
 \be \label{Q4}
 \int_E |g| \ \le \  \frac{p}{p-1}\|g\|_{p,w}m(E)^{1-1/p} \ .
 \ee
 From Corollary 3.1 we have that the function $\xi\ra \pa^{r_j} q(\xi+ai)/\pa\xi_1^{r_j}$  is in $L^{d/r_j}_w(\mathbb{T}^d;\mathcal{L}(\mathbb{C}^d)),     \ j=1,\dots,d-1$, whence the product of the derivatives in (\ref{O4}) is in
 $L^{d/r}_w(\mathbb{T}^d;\mathbb{C})$. From the H\"{o}lder inequality we have that for any $p\ge 1$ and  $1/p+1/p'=1$, 
 \begin{multline} \label{R4}
 \int_{\mathbb{T}^d} \left|     A(\xi+ai,\eta)\prod_{j=1}^{d-1} \frac{\pa ^{r_j}q_{m_j,n_j}(\xi+ai,\eta)}{\pa\xi_1^{r_j}}                   \right| \ d\xi \\
  \le \  \left\{     \int_{\mathbb{T}^d} \left|     A(\xi+ai,\eta)\right|^{p'}    \ d\xi                 \right\}^{1/p'}   \left\{     \int_{\mathbb{T}^d} \left|   \prod_{j=1}^{d-1} \frac{\pa ^{r_j}q_{m_j,n_j}(\xi+ai,\eta)}{\pa\xi_1^{r_j}}  \right|^{p}    \ d\xi                 \right\}^{1/p} \ .
 \end{multline}
The first integral on the  RHS of (\ref{R4}) is bounded provided $1/p'>(d-1-r)/d$, while the second integral is bounded  provided  $1/p >r/d$. Since $r\le d-1$ we are able to choose $p>1$ which satisfies both these inequalities.  It follows that
$|I_{d-1}(x)|\le C/(1+|x|)^{d-1}, \ x\in\mathbb{Z}^d$.  We have proven (\ref{E4}) for $\nu=1$. 

 To show that (\ref{E4}) continues to hold for $0<\nu\le 1$ 
 we do a fractional integration by parts in (\ref{M4}) and use the inequality for  (\ref{CB3}) given by Corollary 3.1. We choose $\De\xi\simeq 1/(1+|x|)$ such that $\De\xi\cdot x=\pi$. We have then that
 \begin{multline} \label{S4}
 2 I_{d-1}(x) \ = \ \frac{1}{(ix_1)^{d-1}}\int_{\mathbb{T}^d} e^{-i\xi\cdot x}  \frac{\pa^{d-1} F(\xi+ai,\eta)}{\pa\xi_1^{d-1}} \  d\xi  \\
 - \frac{1}{(ix_1)^{d-1}}\int_{\mathbb{T}^d} e^{-i\xi\cdot x}  \frac{\pa^{d-1} F(\xi+\De\xi+ai,\eta)}{\pa\xi_1^{d-1}} \  d\xi  \ .
 \end{multline}
 Since $\De\xi\simeq 1/(1+|x|)$  we have from (\ref{S4}) that
 \be\label{T4}
 \left|I_{d-1}(x)\right| \ \le \  \frac{C}{|x|^{d-\nu}} \int_{\mathbb{T}^d}  \left|  \frac{1}{|\De\xi|^{1-\nu} } \left\{\frac{\pa^{d-1} F(\xi+ai,\eta)}{\pa\xi_1^{d-1}}-\frac{\pa^{d-1} F(\xi+\De\xi+ai,\eta)}{\pa\xi_1^{d-1}} \right\}                                \right| \ d\xi  \ .
 \ee
 The integrand in (\ref{S4}) consists of a sum of terms of the form
 \begin{multline} \label{U4}
A(\xi+ai,\eta)\prod_{j=1}^{d-1} \frac{\pa ^{r_j}q_{m_j,n_j}(\xi+ai,\eta)}{\pa\xi_1^{r_j}} -A(\xi+\De\xi+ai,\eta)\prod_{j=1}^{d-1} \frac{\pa ^{r_j}q_{m_j,n_j}(\xi+\De\xi+ai,\eta)}{\pa\xi_1^{r_j}} \\
= \ 
\sum_{k=0}^{d-1} J_k(\xi+ai,\eta,\De\xi) \  ,
\end{multline}
 where
 \begin{multline} \label{V4}
 J_0(\xi+ai,\eta,\De\xi)  \ = \  \{A(\xi+ai,\eta)-A(\xi+\De\xi+ai,\eta)\}\prod_{j=1}^{d-1} \frac{\pa ^{r_j}q_{m_j,n_j}(\xi+ai,\eta)}{\pa\xi_1^{r_j}}  \ , \\
  {\rm For \ } k\ge 1, \quad J_k(\xi+ai,\eta,\De\xi)  \ = \  A(\xi+\De\xi+ai,\eta)\prod_{j=1}^{k-1} \frac{\pa ^{r_j}q_{m_j,n_j}(\xi+\De\xi+ai,\eta)}{\pa\xi_1^{r_j}}  \   \\
\times  \    \left\{         \frac{\pa ^{r_k}q_{m_j,n_j}(\xi+ai,\eta)}{\pa\xi_1^{r_k}}         -     \frac{\pa ^{r_k}q_{m_j,n_j}(\xi+\De\xi+ai,\eta)}{\pa\xi_1^{r_k} }                         \right\}
\prod_{j=k+1}^{d-1} \frac{\pa ^{r_j}q_{m_j,n_j}(\xi+ai,\eta)}{\pa\xi_1^{r_j}}  \ .
 \end{multline}
 To bound the integral of $J_0$ in (\ref{V4}) we write 
 \begin{multline}\label{W4}
     A(\xi+ai,\eta)-A(\xi+\De\xi+ai,\eta)     \\
      = \ A_0(\xi+ai,\eta,\De\xi) + {\rm Tr} \left[A_1(\xi+ai,\eta,\De\xi)\left\{       q(\xi+ai,\eta)-q(\xi+\De\xi+ai,\eta)               \right\} \right] \  .
\end{multline}
 The function $A_0(\xi+ai,\eta,\De\xi)$ and the $d\times d$ matrix $A_1(\xi+ai,\eta,\De\xi)$ satisfy inequalities similar to (\ref{P4}).   Thus we have that
 \be \label{X4}
 \begin{array}{ccc}
 |A_0(\xi+ai,\eta,\De\xi)| \ &\le& \ C\left[  \frac{1}{|\xi|^{d-1-r}} +  \frac{1}{|\xi+\De\xi|^{d-1-r}} \right]\ , \quad \xi\in\mathbb{T}^d \ , \\
|A_0(\xi+ai,\eta,\De\xi)| \ &\le& \ C|\De\xi|\left[  \frac{1}{|\xi|^{d-r}} +  \frac{1}{|\xi+\De\xi|^{d-r}} \right]\ , \quad \xi\in\mathbb{T}^d \ , \\
|A_1(\xi+ai,\eta,\De\xi)| \ &\le& \ C\left[  \frac{1}{|\xi|^{d-1-r}} +  \frac{1}{|\xi+\De\xi|^{d-1-r}} \right]\ , \quad \xi\in\mathbb{T}^d \ .
\end{array}
\ee
Using the first two inequalities in (\ref{X4}), we see that 
\be \label{Y4}
  \int_{\mathbb{T}^d} \left|     \frac{A_0(\xi+ai,\eta,\De\xi)}{|\De\xi|^{1-\nu}}\right|^{p'}    \ d\xi    \ < \ \infty   \quad {\rm provided \ } \frac{1}{p'}>\frac{d-\nu-r}{d} \ .
\ee
It follows from (\ref{Y4}) similarly to  before that the contribution to (\ref{T4}) corresponding to the term in $J_0$ containing  $A_0$ is integrable on $\mathbb{T}^d$  for all $\nu, \  0<\nu\le 1$.  To deal with the term in $J_0$ corresponding to the second term
on the RHS of (\ref{W4}) we observe from Corollary 3.1 that the function $\xi\ra \left\{       q(\xi+ai,\eta)-q(\xi+\De\xi+ai,\eta)               \right\}/|\De\xi|^{1-\nu}$ is in the space $L_w^{d/(1-\nu)}(\mathbb{T}^d;\mathcal{L}(\mathbb{C}^d))$. 
Thus we have that the product of this function together with the other derivatives in the formula for $J_0$ is in $L_w^{d/(r+1-\nu)}(\mathbb{T}^d;\mathcal{L}(\mathbb{C}^d))$.Using the third inequality of (\ref{X4}) and arguing as before 
we conclude that that the term in $J_0$ corresponding to $A_1$ is integrable on $\mathbb{T}^d$  for all $\nu, \  0<\nu\le 1$.  We can similarly see that the contributions to (\ref{T4}) from  $J_k, \ k\ge 1,$ are also  integrable on $\mathbb{T}^d$  for all $\nu, \  0<\nu\le 1$.  
\end{proof}
We may use Proposition 3.1 to obtain bounds on the Fourier inverse of the function $\xi\ra h_m(\xi,\eta), \ \xi\in\mathbb{T}^d$, defined by   (\ref{D3}).
\begin{lem}
The Fourier inverse $x\ra \hat{h}_m(x,\eta), \ x\in\mathbb{Z}^d,$ of the function $h_m$ of  (\ref{D3}) satisfies an inequality,
\be \label{Z4}
| \hat{h}_m(x,\eta) | \ \le \  C_{d,\nu}m^{3d}(1-\la/2)^m \frac{e^{-c_d\sqrt{\la\eta}|x|}}{(1+|x|)^{d-\nu}} \ , \quad x\in\mathbb{Z}^d \ ,  \ 0<\eta\le 1, \ \nu>0 \ ,
\ee
for constants $C_{d,\nu}$ depending only  on $d,\nu$ and $c_d>0$ on $d$.  
\end{lem}
\begin{proof}
We proceed as in the proof of Lemma 4.1, using the representation
\be \label{AA4}
\hat{h}_m(x,\eta) \ = \ \frac{e^{a\cdot x}}{(2\pi)^d}\int_{\mathbb{T}^d}  h_m(\xi+ai,\eta) e^{-i\xi\cdot x} \  d\xi \ ,   \quad x\in\mathbb{Z}^d  \ ,
\ee
to obtain the exponential falloff in (\ref{Z4}).  Then we use integration by parts to obtain a formula similar to (\ref{M4}). Arguing as in the proof of Lemma 4.1, the  inequality (\ref{Z4})  then follows from the bound on  (\ref{BW3}) given in  Proposition 3.1.
\end{proof}
In the rest of this section we shall assume the following hypothesis:
\begin{hyp}
The functions $\hat{h}_m(\cdot,\eta), \ m\ge 1,$ of Lemma 4.2  satisfy an inequality
\be \label{AB4}
| \hat{h}_m(x,\eta) | \ \le \  \frac{C_\al^m}{(1+|x|)^{d+\al}} \ , \quad x\in\mathbb{Z}^d \ ,  \ 0<\eta\le 1 , \ \ m=1,2,\dots \ ,
\ee
for some $\al>0$ and constant $C_\al$. 
\end{hyp}
\begin{lem}
The Fourier inverse $x\ra \hat{q}(x,\eta), \ x\in\mathbb{Z}^d,$ of the function $q$ of  (\ref{D3}) satisfies an inequality,
\be \label{AC4}
| \hat{q}(x,\eta) | \ \le \   \frac{Ce^{-c_d\sqrt{\la\eta}|x|}}{(1+|x|)^{d+\beta}} \ , \quad x\in\mathbb{Z}^d \ ,  \ 0<\eta\le 1, 
\ee
for some $\beta>0$ and constants $C,c_d$. 
\end{lem}
\begin{proof}
From Lemma 4.2 and Hypothesis 4.1 we obtain the inequality,
\be \label{AD4}
| \hat{h}_m(x,\eta) | \ \le \  \left[ C_{d,\nu}m^{3d}(1-\la/2)^m \frac{e^{-c_d\sqrt{\la\eta}|x|}}{(1+|x|)^{d-\nu}}\right]^{1-\del}\left[\frac{C_\al^m}{(1+|x|)^{d+\al}}\right] ^\del\ ,
\ee
for any $\del>0$. We first choose $\del$ such that $(1-\la/2)C^\del_\al<1$ and then $\nu$ such that  $\nu<\al\del/(1-\del)$. The inequality (\ref{AC4}) with $\beta=\al\del-\nu(1-\del)>0$ follows from (\ref{D3}) and (\ref{AD4}). 
\end{proof}
We can use Lemma 4.3 to improve the estimate (\ref{C4}) on the second derivative (\ref{B4})  of the averaged Green's function. 
\begin{proposition}
The second derivative (\ref{B4})  of the averaged Green's function satisfies an inequality
\be \label{AE4}
\left|            \na\na^*G_{\mathbf{a},\eta}(x)         \right| \ \le \   \frac{C_de^{-c_d\sqrt{\la\eta}|x|}}{(1+|x|)^{d}} \ , \quad x\in\mathbb{Z}^d \ ,  \ 0<\eta\le 1,
\ee
for some constants $C_d,c_d>0$. 
\end{proposition}
\begin{proof}
Using analytic continuation again we have the representation
\be \label{AF4}
\na\na^*G_{\mathbf{a},\eta}(x)  \ = \  \frac{e^{a\cdot x}}{(2\pi)^d}\int_{\mathbb{T}^d} \frac{e(\xi+ai)e(\xi-ai)^*e^{-ix\cdot\xi} }{\eta+e(\xi-ai)^*q(\xi+ai,\eta)e(\xi+ai)} \ d\xi \  .
\ee
By choosing $a$ in (\ref{AF4})  appropriately we obtain the exponential falloff in (\ref{AE4}). 
Let $\phi:\mathbb{R}^d\ra\mathbb{R}$ be a $C^\infty_c$ function   such that $\phi(\xi)=1$ for $|\xi|<20\pi$. To obtain the polynomial decay in (\ref{AE4}) we write the integral in (\ref{AF4}) as a sum,
$I_1(x)+I_2(x)$, where
\be \label{AG4}
I_1(x) \ = \ \int_{\mathbb{T}^d} \phi(\{1+|x|\}\xi)\frac{e(\xi+ai)e(\xi-ai)^*e^{-ix\cdot\xi} }{\eta+e(\xi-ai)^*q(\xi+ai,\eta)e(\xi+ai)} \ d\xi \  .
\ee
It easily follows from (\ref{I4}) that $|I_1(x)|\le C/(1+|x|)^d$.  

Assuming  that $|x_1|\ge |x|/d$ and using integration by parts, we write $I_2(x)$ as
\begin{multline} \label{AH4}
I_2(x) \ = \ \frac{1}{(ix_1)^{d}}\int_{\mathbb{T}^d}e^{-ix\cdot\xi} \frac{\pa^{d}}{\pa\xi_1^{d}}\Bigg[ \left\{1-\phi(\{1+|x|\}\xi)\right\}  \\
\times\frac{e(\xi+ai)e(\xi-ai)^* }{\eta+e(\xi-ai)^*q(\xi+ai,\eta)e(\xi+ai)} \Bigg] \ d\xi  \ = \ \sum_{j=0}^d I_{2,j}(x) \  .
\end{multline}
The sum over $j, \ 0\le j\le d,$ in (\ref{AH4}) comes from the application of Leibniz rule, with $I_{2,j}(x)$ being the term with $j$ derivatives acting on  the function $\xi\ra 1-\phi(\{1+|x|\}\xi)$.
Hence we have for $j\ge 1$  that 
\be \label{AI4}
|I_{2,j}(x)| \ \le \  \frac{C}{(1+|x|)^{d-j}} \int_{|\xi|<R/(1+|x|)}  \left|        \frac{\pa^{d-j}}{\pa\xi_1^{d-j}} \left\{\frac{e(\xi+ai)e(\xi-ai)^* }{\eta+e(\xi-ai)^*q(\xi+ai,\eta)e(\xi+ai)}            \right\}                           \right| \ d\xi \ ,
\ee
for any ball in $\mathbb{R}^d$ of radius $R$  containing ${\rm supp} \  \phi(\cdot)$. 
Arguing as in the proof of Lemma 4.1, in particular using the inequality (\ref{Q4}), we see that the integral in (\ref{AI4}) is bounded by $C/(1+|x|)^j$, whence we conclude that $|I_{2,j}(x)|\le C/(1+|x|)^d$ if $j\ge 1$

To  bound $I_{2,0}(x)$ we  write similarly to (\ref{N4}), 
\be \label{AJ4}
  \frac{\pa^{d}}{\pa\xi_1^{d}} \left\{\frac{e(\xi+ai)e(\xi-ai)^* }{\eta+e(\xi-ai)^*q(\xi+ai,\eta)e(\xi+ai)} \right\}     \ = \  P\left(  \xi+ ai,\eta,\frac{\pa q(\xi+ai,\eta)}{\pa\xi_1}, \dots,      \frac{\pa^{d} q(\xi+ai,\eta)}{\pa\xi_1^{d}}                           \right)  \ .
\ee
A typical term in the polynomial  of (\ref{AJ4}) has the form
\be \label{AK4}
A(\xi+ ai,\eta)\prod_{j=1}^{d} \frac{\pa ^{r_j}q_{m_j,n_j}(\xi+ai,\eta)}{\pa\xi_1^{r_j}} \ , \quad {\rm with \ } r_1+\cdots+r_{d}=r \ \le  \ d \ ,
\ee
where the $r_j$ are non-negative integers.  One also sees that
\be \label{AL4}
|A(\xi+ai,\eta)| \ \le \ \frac{C}{|\xi|^{d-r}} \ , \quad \xi\in\mathbb{T}^d \ .
\ee

We may bound the integral of  a term (\ref{AK4}) as before provided $r<d$. To do this we do a fractional integration by parts in the integral
\be \label{AM4}
\frac{1}{(ix_1)^{d}}\int_{\mathbb{T}^d}e^{-ix\cdot\xi} \left\{1-\phi(\{1+|x|\}\xi)\right\}  A(\xi+ ai,\eta)\prod_{j=1}^{d} \frac{\pa ^{r_j}q_{m_j,n_j}(\xi+ai,\eta)}{\pa\xi_1^{r_j}}  \ d\xi  \ ,
\ee
and proceed in a similar way to (\ref{S4})-(\ref{W4}).  Thus as in (\ref{S4}) we do a translation  by $\De\xi$  in the integrand of (\ref{AM4}) where $\De\xi\cdot x=\pi$.  The expression (\ref{AM4})  is half the value of 
\begin{multline} \label{AN4}
\frac{1}{(ix_1)^{d}}\int_{\mathbb{T}^d}e^{-ix\cdot\xi} \left\{\phi(\{1+|x|\}[\xi+\De\xi])-\phi(\{1+|x|\}\xi)\right\}  A(\xi+ ai,\eta)\prod_{j=1}^{d} \frac{\pa ^{r_j}q_{m_j,n_j}(\xi+ai,\eta)}{\pa\xi_1^{r_j}}  \ d\xi  \\
+\frac{1}{(ix_1)^{d}}\int_{\mathbb{T}^d}e^{-ix\cdot\xi} \left\{1-\phi(\{1+|x|\}[\xi+\De\xi])\right\}  \\
\times  \left[ A(\xi+ ai,\eta)\prod_{j=1}^{d} \frac{\pa ^{r_j}q_{m_j,n_j}(\xi+ai,\eta)}{\pa\xi_1^{r_j}}  -    A(\xi+\De\xi+ ai,\eta)\prod_{j=1}^{d} \frac{\pa ^{r_j}q_{m_j,n_j}(\xi+\De\xi+ai,\eta)}{\pa\xi_1^{r_j}}\right]  \ d\xi  \ .
\end{multline}
The integration in the  first term of (\ref{AN4}) is confined to  the region $\{|\xi|<M/(1+|x|)\}$ for some constant $M>20\pi$.   Since we may also choose $\De\xi$ to satisfy $|\De\xi|<\pi/(1+|x|)$,  we see that
$\phi(\{1+|x|\}[\xi+\De\xi])-\phi(\{1+|x|\}\xi)=0$ if $|\xi|<10\pi/(1+|x|)$.  Hence we have from (\ref{AL4}) that the first term in (\ref{AN4}) is bounded by
\be \label{AO4}
\frac{C}{(1+|x|)^r} \int_{|\xi|<M/(1+|x|)} \left|                 \prod_{j=1}^{d} \frac{\pa ^{r_j}q_{m_j,n_j}(\xi+ai,\eta)}{\pa\xi_1^{r_j}}                \right| \ d\xi  \ .
\ee
Applying the inequality (\ref{Q4}) in (\ref{AO4}) and using Corollary 3.1 we conclude that the first integral in (\ref{AN4}) is bounded by $C/(1+|x|)^d$. 

 To bound the second term in (\ref{AN4}) we use the formula  (\ref{U4}), replacing $d-1$ by $d$.   We consider the contribution   from  $J_k, $ as in (\ref{V4}), with $k\ge1$.   Thus we need to estimate an integral,
\be \label{AP4}
\frac{1}{(1+|x|)^d} \int_{|\xi|>1/(1+|x|)}  |J_k(\xi+ai,\eta,\De\xi)| \ \frac{d\xi}{|\xi|^{d-r}} \ . 
\ee
We have from Corollary 3.1 that the function $\xi\ra J_k(\xi+ai,\eta,\De\xi)/|\De\xi|^{1-\nu}$ is in $L^{d/(r+1-\nu)}_w(\mathbb{T}^d; \mathbb{C})$. It follows then from (\ref{Q4}) that  the expression in (\ref{AP4}) is bounded by
\begin{multline} \label{AQ4}
\frac{C}{(1+|x|)^{r+1-\nu}} \sum_{n\ge 1} 2^{-(d-r)n}\int_{2^{n-1}/(1+|x|)<|\xi|<2^n/(1+|x|)}  \frac{|J_k(\xi+ai,\eta,\De\xi)|}{|\De\xi|^{1-\nu}} \ d\xi \\
\le \ \frac{C}{(1+|x|)^d}\sum_{n\ge 0} 2^{-(1-\nu)n} \ \le \ \ \frac{C}{(1+|x|)^d} \ .
\end{multline}
The $J_0$ term may be similarly  bounded  following the representation   (\ref{W4}) and inequalities (\ref{X4})  with  $d-1$ replaced by $d$. 

Finally we need to deal with the case when $r=d$ in (\ref{AK4}). Thus we need to bound the integral
\begin{multline} \label{AR4}
I_3(x) \ = \ \frac{1}{(ix_1)^{d}}\int_{\mathbb{T}^d}e^{-ix\cdot\xi}  \left\{1-\phi(\{1+|x|\}\xi)\right\}   A(\xi+ ai,\eta)B(\xi+ai,\eta) \  d\xi \ , \\
 {\rm where \ } B(\xi+ai,\eta) \ = \ \prod_{j=1}^{d} \frac{\pa ^{r_j}q_{m_j,n_j}(\xi+ai,\eta)}{\pa\xi_1^{r_j}}  \ d\xi \  \ .
\end{multline}
 We first estimate (\ref{AR4}) in the case $\phi(\cdot)\equiv0$.  Note that the function $\xi\ra A(\xi,\eta)$ is of the form (\ref{D4}), whence  by Lemma 4.1 its Fourier inverse $x\ra\hat{A}(x,\eta)$ satisfies an inequality (\ref{E4}).    We have then by Parseval's theorem that
 \be \label{AS4}
 \frac{e^{a\cdot x}}{(2\pi)^d} I_3(x) \ = \ \frac{1}{(ix_1)^{d}}\sum_{y\in\mathbb{Z}^d} \hat{A}(y,\eta)\hat{B}(x-y,\eta)  \ .
 \ee
 We also have from Lemma 4.3 that
 \be \label{AT4}
| \hat{B}(y,\eta) | \ \le \   \frac{Ce^{-c_d\sqrt{\la\eta}|y|}}{(1+|y|)^{\beta}} \ , \quad y\in\mathbb{Z}^d \ ,  \ 0<\eta\le 1.
\ee
The inequality (\ref{AT4})  follows immediately from Lemma 4.3 when $r_1=d$ in (\ref{AR4}),  and otherwise by estimating convolutions.  For the bound on the function $x\ra\hat{A}(x,\eta)$ we choose $\nu$ in (\ref{E4}) such that $\nu<\beta$, to conclude that
\be \label{AU4}
\left|                  \sum_{y\in\mathbb{Z}^d} \hat{A}(y,\eta)\hat{B}(x-y,\eta)           \right| \ \le \ Ce^{-c_d\sqrt{\la\eta}|x|} \ , \quad x\in\mathbb{Z}^d \ ,
\ee
for some constant $C$. 

We also use the Parseval identity to estimate the term involving $\phi(\cdot)$ in (\ref{AR4}).  We have that 
\begin{multline} \label{AV4}
\frac{1}{(2\pi)^d}I_4(x) \ = \ \frac{1}{(2\pi ix_1)^{d}}\int_{\mathbb{T}^d}e^{-ix\cdot\xi} \phi(\{1+|x|\}\xi)   A(\xi+ ai,\eta)B(\xi+ai,\eta)  \ d\xi \\
= \ \frac{1}{(ix_1)^{d}}\sum_{y\in\mathbb{Z}^d}  \frac{1}{(1+|x|)^d} \hat{\phi}\left(\frac{x-y}{1+|x|}\right) \hat{C}(y,\eta,a) \ , \quad {\rm where \ } C(\xi,\eta,a)  \ = \  A(\xi+ ai,\eta)B(\xi+ai,\eta)  \ .
\end{multline}
We may estimate the function $y\ra \hat{C}(y,\eta,a)$ similarly to (\ref{AU4}).  Choosing $|a|=c_d\sqrt{\la\eta}$ with $c_d>0$ sufficiently small, we see  from Lemma 4.1 and Lemma 4.3 that $|\hat{C}(y,\eta,a)|\le K/(1+|y|)^{\beta-\nu}, \ y\in\mathbb{Z}^d$, for some constant $K$.  Since $\hat{\phi}(\cdot)$ is a Schwarz function it is easy to see that the summation over $y\in\mathbb{Z}^d$  in (\ref{AV4}) is bounded, independent of $x\in\mathbb{Z}^d$. We conclude that $|I_4(x)|\le C/(1+|x|)^d, \ x\in\mathbb{Z}^d$, for some constant $C$. 
\end{proof}
\begin{rem}
Proposition 4.1 with $\eta=0$  has been obtained for general environments in \cite{cgo}. The idea of the proof is to use the fact that the function $x\ra \na_y G_{\mathbf{a},
\eta}(x,y,\om), \ x\ne y$  is $\mathbf{a}$-harmonic when $\eta=0$.  This follows by differentiating (\ref{D1}) with respect to $y$. One then uses the Cacciapoli inequality to estimate the function  $x\ra \na_x\na_y G_{\mathbf{a},
\eta}(x,y,\om)$ in terms of the function $x\ra \na_y G_{\mathbf{a},\eta}(x,y,\om)$.
\end{rem}
\begin{corollary}
The second derivative (\ref{B4})  of the averaged Green's function is H\"{o}lder continuous. For some $\al>0$ it  satisfies an inequality
\begin{multline} \label{AW4}
\left|            \na\na^*G_{\mathbf{a},\eta}(x)  -  \na\na^*G_{\mathbf{a},\eta}(x')        \right| \ \le \   \frac{C_d|x-x'|^\al e^{-c_d\sqrt{\la\eta}|x|}}{(1+|x|)^{d+\al}} \ , \\
 {\rm where \ } x,x'\in\mathbb{Z}^d  \quad {\rm with \ }|x-x'|<|x|/2\ ,  \ 0<\eta\le 1, \quad {\rm and \ } C_d,c_d \ {\rm constants.}
\end{multline}
\end{corollary}
\begin{proof}
We use the fact that $|e^{-ix\cdot\xi}-e^{-ix'\cdot\xi}|\le 2|x-x'|^\al|\xi|^\al$ for any $\al, \ 0<\al\le 1$ and apply the argument of Proposition 4.1. The  H\"{o}lder continuity of the term (\ref{AS4})  follows from the  H\"{o}lder continuity of the
function $x\ra\hat{A}(x,\eta)$, which one can derive from the argument of  Lemma 4.1. In fact as well as the inequality (\ref{E4}) one has the inequality
\be \label{AX4}
| \hat{F}(x,\eta)- \hat{F}(x',\eta)| \ \le \   \frac{C_{d,\nu}|x-x'|e^{-c_d\sqrt{\la\eta}|x|}}{(1+|x|)^{d-\nu}} \ , \quad x,x'\in\mathbb{Z}^d \ ,  \ 0<\eta\le 1, \ \nu>0 \ ,
\ee
\end{proof}
\begin{proof}[Proof of Theorem 1.1]
We follow the argument in $\S7$ of \cite{cs1}, using Hypothesis 3.1 and Corollary 4.1.

\end{proof}

\vspace{.1in}
\section{Independent Variable Environment}
We first establish a covariance formula for functions of independent variables. Let $\psi_1,\dots,\psi_N$ be $N$ independent variables  taking values in $\mathbb{R}^n$ and $G(\cdot)$ be a $C^1$  function on $\mathbb{R}^{nN}$. 
We write
\begin{multline} \label{A5}
G(\psi_1,\dots,\psi_N)- \langle G(\psi_1,\dots,\psi_N)\rangle \ = \ \sum_{k=1}^N \{G_k(\psi_1,\dots,\psi_k)- G_{k-1}(\psi_1,\dots,\psi_{k-1})  \ \}  \ , \\
{\rm where \ } \quad G_k(\psi_1,\dots,\psi_k) \ = \ E[G(\psi_1,\dots,\psi_N) \ \big| \ \psi_1,\dots,\psi_k] \ , \quad k=0,\dots, N \ .
\end{multline}
The identity (\ref{A5}) is sometimes known as the Martingale method  (see Chapter 3 of \cite{vh}). 
We see from (\ref{A5}) that
\begin{multline} \label{B5}
{\rm cov}\left[  F(\psi_1,\dots, \psi_N), G(\psi_1,\dots, \psi_N)   \right] \\
 = \  \sum_{k=1}^N E\left[        \{F_k(\psi_1,\dots,\psi_k)- F_{k-1}(\psi_1,\dots,\psi_{k-1})\}\{G_k(\psi_1,\dots,\psi_k)- G_{k-1}(\psi_1,\dots,\psi_{k-1})\}            \right]  \ .
\end{multline} 
Next we write
\begin{multline} \label{C5}
G_k(\psi_1,\dots,\psi_k)- G_{k-1}(\psi_1,\dots,\psi_{k-1}) \ = \ \{   G_k(\psi_1,\dots,\psi_k)-G_k(\psi_1,\dots,\psi_{k-1}, \langle\psi_k\rangle)  \} \\
- \ E\left[      \{   G_k(\psi_1,\dots,\psi_k)-G_k(\psi_1,\dots,\psi_{k-1}, \langle\psi_k\rangle)  \}  \ \big| \ \psi_1,\dots, \psi_{k-1}                  \right] \ .
\end{multline}
Now we use (\ref{C5}) and   FTC to obtain  the formula
\begin{multline} \label{D5}
G_k(\psi_1,\dots,\psi_k)- G_{k-1}(\psi_1,\dots,\psi_{k-1}) \\
 = \ \int_0^1 dt \ \{G_k(t,\psi_1,\dots,\psi_k)-E\left[        G_k(t,\psi_1,\dots,\psi_k) \ \big| \ \psi_1,\dots,\psi_{k-1}           \right]  \ \} \ , \\
G_k(t,\psi_1,\dots,\psi_k)  \ = \  D_{\psi_k} G_k(\psi_1,\dots,\psi_{k-1}, t\psi_k+(1-t)\langle\psi_k\rangle)^*  \{\psi_k-\langle\psi_k\rangle\} \ ,
\end{multline}
where $D_{\psi_k}G(\cdot)$ denotes the gradient of $G(\cdot)$ with respect to $\psi_k$.  The  identities (\ref{B5})-(\ref{D5}) give a formula for the covariance of $F(\cdot)$ and $G(\cdot)$ in terms of the gradients of these functions. 
Using this formula we can easily prove a Poincar\'{e} type inequality due to Gloria and Otto \cite{go}  for functions of independent variables:
 \begin{lem} [Lemma 2.3 of \cite{go}]
 Let $a=\{a_i: \ 1\le i\le N\}$ be a sequence of  i.i.d. random variables with values in $\mathbb{R}$,  and  $X$ a $C^1$ function  of $a\in\mathbb{R}^N$.  Then
 \be \label{E5}
 {\rm var}[X] \ \le \ \left\langle      \sum_ {i=1}^N \sup_{a_i}\left|         \frac{\pa X}{\pa a_i}        \right|^2             \right\rangle {\rm var}[a_1] \ .
 \ee
 \end{lem} 
 \begin{proof}
 Setting $G(\cdot)\equiv X(\cdot)$ and $\psi_i=a_i, \ 1\le i\le N,$ we have from (\ref{B5})-(\ref{D5}) that
 \begin{multline} \label{F5}
  {\rm var}[X] \  = \  \sum_{i=1}^N  \int_0^1 \int_0^1 ds \ dt \  E\Bigg[              \big\{G_k(s,\psi_1,\dots,\psi_k)-E\left[        G_k(s,\psi_1,\dots,\psi_k) \ \big| \ \psi_1,\dots,\psi_{k-1}           \right]  \ \big\}        \\
   \times   \left\{G_k(t,\psi_1,\dots,\psi_k)-E\left[        G_k(t,\psi_1,\dots,\psi_k) \ \big| \ \psi_1,\dots,\psi_{k-1}           \right]  \ \right\}                \Bigg] \  .
 \end{multline}
 Applying the Schwarz inequality to (\ref{F5}) we have that
 \begin{multline} \label{G5}
   {\rm var}[X] \  \le  \  \sum_{i=1}^N  \int_0^1 dt  \ E\Bigg[              \big\{G_k(t,\psi_1,\dots,\psi_k)-E\left[        G_k(t,\psi_1,\dots,\psi_k) \ \big| \ \psi_1,\dots,\psi_{k-1}           \right]  \ \big\}^2              \Bigg]  \\
   \le \ \sum_{i=1}^N  \int_0^1 dt  \ E\Big[              G_k(t,\psi_1,\dots,\psi_k)^2             \Big] \ .
 \end{multline}
 We see from (\ref{D5}) that 
 \be \label{H5}
   \left|G_k(t,\psi_1,\dots,\psi_k)\right| \ \le \  \sup_{\psi_k}\left|\frac{\pa G_k(\psi_1,\dots,\psi_k)}{\pa \psi_k}\right| |\psi_k-\langle\psi_k\rangle|\ .
 \ee
 Since the variables $\psi_i$ are independent, we conclude from (\ref{H5}) that
 \begin{multline} \label{I5}
 E\Big[              G_k(t,\psi_1,\dots,\psi_k)^2             \Big]  \ \le \ \left\langle          \sup_{\psi_k}\left|\frac{\pa G_k(\psi_1,\dots,\psi_k)}{\pa \psi_k}\right|^2                \right\rangle {\rm var}[\psi_k] \\
 \le \ \left\langle          \sup_{\psi_k}\left|\frac{\pa G(\psi_1,\dots,\psi_N)}{\pa \psi_k}\right|^2                \right\rangle {\rm var}[\psi_k] \ .
 \end{multline}
 The  last inequality in  (\ref{I5}) follows from  the fact that
 \begin{multline} \label{J5}
 \frac{\pa G_k(\psi_1,\dots,\psi_k)}{\pa \psi_k} \ = \ E\left[        \frac{\pa G(\psi_1,\dots,\psi_N)}{\pa \psi_k}    \ \big| \      \psi_1,\dots,\psi_k                                        \right]  \ , \\
  \left|\frac{\pa G_k(\psi_1,\dots,\psi_k)}{\pa \psi_k}\right|^2 \ \le \ E\left[    \left|    \frac{\pa G(\psi_1,\dots,\psi_N)}{\pa \psi_k} \right|^2   \ \Big| \      \psi_1,\dots,\psi_k                                        \right]  \ .
 \end{multline}
 The inequality (\ref{E5}) is now a consequence of  (\ref{G5}), (\ref{I5}).
 \end{proof}
 We show next  how the covariance formula (\ref{B5})-(\ref{D5}) can be used to establish Hypothesis 4.1 in the case of  the i.i.d. variable  environment.  Similarly to how we derived the formula (\ref{N3}), we see from 
 (\ref{B3}), (\ref{D3})  that   $\hat{h}_m(x,\eta)$ is the expectation value of the expression, 
 \begin{multline} \label{K5}
 \hat{h}_m(x,\om,\eta) \ = \ \sum_{y_1,\dots, y_{m-1}\in\mathbb{Z}^d} \mathbf{b}(\tau_{-y_0}\om) \prod_{j=1}^m P \  \na\na^*G_\eta(y_j-y_{j-1})\mathbf{b}(\tau_{-y_j}\om)  \ , \\
 {\rm where \ } y_0=0, \ y_m=x \ .
 \end{multline}
 
 It follows from results on weighted norm inequalities (see theorems 1,2 of chapter V of \cite{stein1})  that  the convolution  operator $S_\eta$  on functions $f:\mathbb{Z}^d\ra\mathbb{C}^d$  with kernel $\na\na^*G_\eta(\cdot)$ is bounded on weighted spaces.
 For $1\le p<\infty, \ k\ge0,$ we define the $\ell_{p,k}$ norm of  functions $f:\mathbb{Z}^d\ra\mathbb{C}^d$  by
\be \label{L5}
\| f(\cdot)\|^p_{p,k} \ = \ \sum_{x\in\mathbb{Z}^d}  [1+|x|^k]^p|f(x)|^p \  ,\quad 1\le p<\infty,  \ k>0,
\ee
and the  corresponding space of functions with finite norm (\ref{L5})  by $\ell_{p,k}(\mathbb{Z}^d;\mathbb{C}^d)$.
\begin{theorem} {\rm (weighted Calderon-Zygmund)}
For $1<p<\infty, \ 0<\eta\le 1,$ the operator $S_\eta$ is bounded on $\ell_{p,k}(\mathbb{Z}^d;\mathbb{C}^d)$ provided $kp'<d$.  The operator norm satisfies an inequality
\be \label{M5}
\left\| S_\eta \ \right\|_{\ell_{p,k}(\mathbb{Z}^d;\mathbb{C}^d)} \ \le \  \kappa_{p,k} \ ,
\ee
where $\kappa_{p,k}\ge1$ is independent of $\eta>0$.  
\end{theorem}
It follows from (\ref{K5}) and Theorem 5.1 that
\be \label{N5}
|v^* \hat{h}_m(x,\om,\eta)| \ \le     \  \frac{2^m(1-\la)^{m+1}\kappa_{p,k}^m|v|}{(1+|x|)^k} \ , \quad x\in\mathbb{Z}^d \ ,  \ v\in\mathbb{C}^d, \ \ 0<\eta\le 1.
\ee
By choosing $p$ large enough in (\ref{N5}) we can take $k<d$ arbitrarily close to $d$. 
Bourgain \cite{b} gave an alternative proof of the inequality (\ref{N5}) using a slight variation of the classical CZ theorem i.e. Theorem 5.1 with $k=0$:
\begin{corollary}
For $R\ge 1$ let $\chi_R:\mathbb{Z}^d\ra\{0,1\}$ be the characteristic function of the set $\{x\in\mathbb{Z}^d: \ |x|<R\}$. The convolution  operator $S_\eta^R$  on functions $f:\mathbb{Z}^d\ra\mathbb{C}^d$  with kernel $\na\na^*G_\eta(\cdot)\chi_R(\cdot)$ is bounded on $\ell_{p,0}(\mathbb{Z}^d;\mathbb{C}^d)$  and 
\be \label{O5}
\left\| S_\eta^R \ \right\|_{\ell_{p,0}(\mathbb{Z}^d;\mathbb{C}^d)} \ \le \  \kappa_{p,0} +\kappa \ ,
\ee
where the constant $\kappa$ is independent of $p$ and also $\eta>0,R\ge 1$.
\end{corollary}
\begin{proof}
Let $f_{\eta,R}:\mathbb{Z}^d\ra\mathcal{L}(\mathbb{C}^d)$ be the function $f_{\eta,R}(x)=\na\na^*G_\eta(x)\chi_R(x)-\na\na^*G_{\eta+1/R^2}(x), \ x\in\mathbb{Z}^d$. Since 
$\|\na\na^*G_{\eta+1/R^2}(x)\|_{\mathcal{L}(\mathbb{C}^d)}\le C\exp[-c\sqrt{\eta+1/R^2} \ |x|]/(1+|x|)^d,  \  x\in\mathbb{Z}^d$, for some positive constants $C,c$ we see that 
$\sum_{|x|>R} \|f_{\eta,R}(x)\|_{\mathcal{L}(\mathbb{C}^d)}\le C_1$ for some constant $C_1$ independent of $\eta,R$.  By going into the Fourier representation we also see that
\be \label{P5}
\| \na\na^*G_\eta(x)-\na\na^*G_{\eta+1/R^2}(x)\|_ {\mathcal{L}(\mathbb{C}^d)} \ \le \ \frac{C_2}{R^2(1+|x|)^{d-2}} \ , \quad {\rm if \ }  |x|<R \ ,
\ee
for some constant $C_2$ independent of $\eta,R$.  The inequality (\ref{P5}) holds if $d\ge 3$, but one can prove  a similar inequality when $d=2$.  We conclude from (\ref{P5})  that
$\sum_{|x|<R} \|f_{\eta,R}(x)\|_{\mathcal{L}(\mathbb{C}^d)}\le C_2$ for some constant $C_2$ independent of $\eta,R$.   Hence we have  that $\sum_{x\in\mathbb{Z}^d} \|f_{\eta,R}(x)\|_{\mathcal{L}(\mathbb{C}^d)}\le C_1+C_2$.
 The inequality (\ref{O5}) follows from this and Theorem 5.1.
\end{proof}
\begin{rem}
Note that $\kappa_{p,0}$ converges to $\infty$ for $p\ra 1$ or $p\ra\infty$. 
\end{rem}
 \begin{lem} [Bourgain \cite{b}]
  The function  $x\ra \hat{h}_m(x,\om,\eta)$  satisfies the inequality
\be \label{Q5}
|v^* \hat{h}_m(x,\om,\eta)| \ \le     \  \frac{C_d2^mm^{2d+1}(1-\la)^{m+1}\{\kappa_{p,0}+\kappa\}^{m-1}|v|}{(1+|x|)^k} \ , \quad x\in\mathbb{Z}^d \ ,  \ v\in\mathbb{C}^d, \ \ 0<\eta\le 1,
\ee
where $1/p-1/2=k/2d$ and  $C_d$ depends only on $d$. 
\end{lem} 
\begin{proof}
We consider the disjoint regions of points $[y_0,\dots,y_m], \ y_0=0,y_m=x, \ |x|=mR,$ with $y_j\in\mathbb{Z}^d$  such that  $2^nR\le \max\{|y_j-y_{j-1}|: \ j=1,2,\dots,m\} < 2^{n+1}R, \ n=0,1,\dots$, which we denote $S_{x,n,m}$.  We can associate a stopping time $\tau$  to points in $S_{x,n,m}$  by setting $\tau=\min\{j:  \ |y_j-y_{j-1}|\ge 2^nR\}$. Then we have that
\begin{multline} \label{R5}
\sum_{[y_1,\dots y_{m-1}]\in S_{x,n,m}} \  v^*\mathbf{b}(\tau_{-y_0}\om) \prod_{j=1}^m P \  \na\na^*G_\eta(y_j-y_{j-1})\mathbf{b}(\tau_{-y_j}\om)      \\
= \  \sum_{\tau=1}^m \sum_{y_1,\dots y_{m-1}\in \mathbb{Z}^d} \  v^*\mathbf{b}(\tau_{-y_0}\om) \prod_{j=1}^{\tau-1} P \  \na\na^*G_\eta(y_j-y_{j-1})\chi_{2^{n}R}(y_j-y_{j-1})\mathbf{b}(\tau_{-y_j}\om) 
 \\ \times \ P \  \na\na^*G_\eta(y_\tau-y_{\tau-1})  \{\chi_{2^{n+1}R}-\chi_{2^nR}\}(y_\tau-y_{\tau-1})\mathbf{b}(\tau_{-y_\tau}\om)  \\
  \times  \  \prod_{j=\tau+1}^m  P \  \na\na^*G_\eta(y_j-y_{j-1})\chi_{2^{n+1}R}(y_j-y_{j-1}) \mathbf{b}(\tau_{-y_j}\om)    \ .
\end{multline} 
Let $\om_j\in\Om, \ j=0,\dots,m,$ be arbitrary and  consider the kernels,
\be \label{S5}
\begin{array}{ccc}
K_1(y_{\tau-1}) \ &=& \ \sum_{y_1,\dots y_{\tau-2}\in \mathbb{Z}^d}    \  \mathbf{b}(\tau_{-y_0}\om_0) \prod_{j=1}^{\tau-1} \na\na^*G_\eta(y_j-y_{j-1})\chi_{2^{n}R}(y_j-y_{j-1})\mathbf{b}(\tau_{-y_j}\om_j)    \ , \\
K_2(y_\tau,x)  \ &=& \ \sum_{y_{\tau+1},\dots y_{m-1}\in \mathbb{Z}^d}\mathbf{b}(\tau_{-y_\tau}\om_\tau) \prod_{j=\tau+1}^m \na\na^*G_\eta(y_j-y_{j-1})\chi_{2^{n+1}R}(y_j-y_{j-1})\mathbf{b}(\tau_{-y_j}\om_j)      \ .
\end{array}
\ee
Applying  Corollary 5.1 we have  that
\be \label{T5}
\begin{array}{ccc}
\|v^*K_1(\cdot)\|_{\ell_{p,0}(\mathbb{Z}^d;\mathbb{C}^d)} \ &\le& \ (1-\la)^\tau\{\kappa_{p,0}+\kappa\}^{\tau-1} |v| \ , \\
 \|K_2(\cdot,x)v\|_{\ell_{p,0}(\mathbb{Z}^d;\mathbb{C}^d)}  \ &\le& \  (1-\la)^{m+1-\tau}\{\kappa_{p,0}+\kappa\}^{m-\tau} |v|\  \ .
\end{array}
\ee
Note that $K_1(y)=0$ for $|y|>(\tau-1) 2^nR$. We also have that $K_2(y,x)=0$ for $|y-x|>(m-\tau) 2^{n+1}R$. Using this fact and (\ref{T5}) we are able to estimate  the  norm of the sum in (\ref{R5}) for given $\tau$.
From H\"{o}lder's inequality and (\ref{T5})  we see that for any $p>1$ it  is  bounded as
\begin{multline} \label{U5}
\frac{C_d2^m}{(2^nR)^d}\sum_{y_{\tau-1},y_\tau} |v^*K_1(y_{\tau-1})| \ \|K_2(y_\tau,x)\|_{\mathcal{L}(\mathbb{C}^d)}  \\
 \le \ \frac{C_d'2^m(1-\la)^{m+1}\{\kappa_{p,0}+\kappa\}^{m-1}}{(2^nR)^d}\ (2^nmR)^{2d/p'} \ , \quad {\rm where \ } \frac{1}{p}+\frac{1}{p'} \ = \ 1 \ .
\end{multline}
Evidently if $p'>2$ we can sum (\ref{U5}) over $n\ge 0$ and the sum is bounded by a constant times   $|v|R^{-d(1-2/p')}$. The inequality (\ref{Q5}) follows by choosing $p'>1$ to satisfy $d(1-2/p')=k$. 
\end{proof}
For $r_1=0,\cdots \tau-1$, we define the $d\times d$ matrix random variable $\om\ra K_{1,n,\tau}(y_{\tau-1},r_1,\om)$ identical to $K_1(y_{\tau-1})$ in (\ref{S5}) but with $\om_j=\om$ for $r_1\le j\le\tau-1$.  Thus $\om_0,\cdots\om_{r_1-1}$ are arbitrary constants in $\Om$. 
Similarly we define for $r_2=\tau,\cdots m$,  the $d\times d$ matrix random variable $\om\ra K_{2,n,\tau}(y_{\tau},x,r_2,\om)$ identical to $K_2(y_{\tau},x)$ in (\ref{S5}) but with $\om_j=\om$ for $\tau\le j\le r_2$.  
We assume  $\om_{r_2+1},\cdots\om_{m}$ are arbitrary constants in $\Om$. 
One can easily see that $\hat{h}(x,\eta)$ is a sum of averages with respect to the $\om_j, \ j<r_1, \ j>r_2$, of covariances
\be \label{V5}
{\rm cov}\left[     K_{1,n,\tau}(y_{\tau-1}, r_1,\cdot),        \ \na\na^*G_\eta(y_\tau-y_{\tau-1})  \{\chi_{2^{n+1}R}-\chi_{2^nR}\}(y_\tau-y_{\tau-1})\          K_{2,n,\tau}(y_\tau, x,r_2,\cdot)   \right] \ .
\ee
We can obtain a bound on $\hat{h}(x,\eta)$ by summing the absolute values of the  at most $2^m$ covariances (\ref{V5}) over $\tau$ with $1\le \tau\le m$ and $n=0,1,\dots$. 
\begin{lem}
Denote the entries of the matrix  $K_{1,n,\tau}(y_{\tau-1}, r_1,\om)$ by  $\{ K^{i,j}_{1,n,\tau}(y_{\tau-1}, r_1,\om)\}, \ 1\le i,j\le d,$  and with a similar notation for the  entries of the matrix $K_{2,n,\tau}(y_\tau, x,r_2,\om)$. Then
there is an inequality, 
\begin{multline} \label{W5}
\sum_{ y_{\tau-1},y_\tau\in\mathbb{Z}^d} \Big| {\rm cov}\left[     K^{i,j}_{1,n,\tau}(y_{\tau-1}, r_1,\cdot), \          K^{i',j'}_{2,n,\tau}(y_\tau, x,r_2,\cdot)   \right] \    \ \Big| \\
 \le \  \frac{C_dm^{2d+3}(1-\la)^{m+1}\{\kappa_{p,0}+\kappa\}^{m-1}(2^nmR)^{2d/p'} }{R^{2d/p-d}} \ .
\end{multline}
\end{lem}
\begin{proof}
We apply the covariance formula (\ref{B5}), (\ref{D5}) and then use deterministic bounds.   Thus we are able to write
\be \label{X5}
{\rm cov}\left[     K^{i,j}_{1,n,\tau}(y_{\tau-1}, r_1,\cdot), \          K^{i',j'}_{2,n,\tau}(y_\tau, x,r_2,\cdot)   \right] \ = \  \sum_{j_1=r_1}^{\tau-1}\sum_{j_2=\tau}^{r_2} \sum_{y\in\mathbb{Z}^d}\left\langle \  h(y, y_{\tau-1},y_\tau, j_1,j_2,\cdot) \ \right\rangle \ ,
\ee
where the function $\om\ra h(y, y_{\tau-1},y_\tau, j_1,j_2,\om)$  corresponds to the gradient operators being applied to  $ \mathbf{b}(\tau_{-y_j}\om)$ with $j=j_1,j_2$ in the formulae (\ref{S5}).  Hence  we have $y_{j_1}=y_{j_2}=y$.   We may write
\begin{multline} \label{Y5}
h(y, y_{\tau-1},y_\tau, j_1,j_2,\om) \ = \  {\rm a \ sum \ of  \ } d^4 {\rm \ terms \ of \ the \ form \ }  \\
f_1(y)g_1(y,y_{\tau-1})g_2(y_\tau,y)f_2(y,x) \ , \\
{\rm where \ } f_1(\cdot), g_1(\cdot, y_{\tau-1})),g_2(y_\tau,\cdot),f_2(\cdot,x) \quad {\rm take \  values \ in \ } \mathbb{R} \ .
\end{multline}
Applying the H\"{o}lder inequality as in (\ref{U5}), we see that
\be \label{Z5}
\begin{array}{ccc}
\sum_{y_{\tau-1}\in\mathbb{Z}^d} |g_1(y, y_{\tau-1}))| \ &\le& \ C_d(1-\la)^{\tau-j_1}\{\kappa_{p,0}+\kappa\}^{\tau-j_1}(2^nmR)^{d/p'} \ , \\
\sum_{y_{\tau}\in\mathbb{Z}^d} |g_2( y_{\tau},y))| \ &\le& \ C_d(1-\la)^{j_2+1-\tau}\{\kappa_{p,0}+\kappa\}^{j_2-\tau}(2^nmR)^{d/p'} \ .
\end{array}
\ee
We also have using Bourgain's method  in the proof of  Lemma 5.2 that
\be \label{AA5}
\sum_{y\in\mathbb{Z}^d} |f_1(y)||f_2(y,x)| \ \le \   \frac{C_dm^{2d+1}(1-\la)^{j_1+m-j_2}\{\kappa_{p,0}+\kappa\}^{j_1+m-j_2}}{R^{2d/p-d}} \ .
\ee
The inequality (\ref{W5}) follows from (\ref{Z5}), (\ref{AA5}). 
 \end{proof}
\begin{proof}[Proof of Hypothesis 4.1]
We sum the covariances (\ref{V5}) over $\tau$ with $1\le \tau\le m$ and $n=0,1,\dots$ using Lemma 5.3. We see from this that the inequality (\ref{AB4}) holds for any $\al<d$. 
\end{proof}
\begin{proof}[Proof of Hypothesis 3.1]
This was already proved in \cite{cs1} in the case of $\mathbf{a}(\tau_x\cdot)= (1+\ga Y_x)I_d$ and  $\mathbf{b}(\tau_x\cdot)=\ga Y_xI_d$, where $Y_x, \ x\in\mathbb{Z}^d,$ are i.i.d.  Bernoulli variables, $Y_x=\pm 1$ with mean $0$.  The same method can be extended to more general i.i.d. variables $\mathbf{b}(\tau_x\cdot), \ x\in\mathbb{Z}^d,$ as was set out in \cite{cn1}. 
Lemma 5.1 can be used to replace the argument in $\S5$ of  \cite{cs1}, which operates through an  identification of the Bernoulli variable space  $L^2(\Om)$  with a Fock space.  In this way a proof of Hypothesis 3.1 for i.i.d. variables can be constructed
which proceeds along similar lines to  the proofs for Euclidean field environments  in  sections 5,6 of \cite{cs1}

We consider fields  $Y:\mathbb{Z}^d\ra\{-1,1\}$ and functions $F(Y(\cdot))$ of the field, which are holomorphic functions  of $Y(x), \ x\in\mathbb{Z}^d$.   Since $Y(\cdot)^2=1$  we may choose a representation of $F(\cdot)$ which has the property
$\pa^2 F(Y(\cdot))/\pa Y(x)^2=0, \ x\in\mathbb{Z}^d$.  For such functions $F(\cdot)$,  Lemma 5.1  implies a Poincar\'{e} inequality in the i.i.d. environment. Thus we have
\be \label{AB5}
{\rm var}[F(Y(\cdot)] \ \le \   \ \sum_{x\in\mathbb{Z}^d}  \left\langle     |dF(x,Y(\cdot))    |^2             \right\rangle  \  ,
\ee
where $dF(x,Y(\cdot))=\pa F(Y(\cdot))/\pa Y(x), \ x\in\mathbb{Z}^d,$ is the gradient field for $F(\cdot)$. 
The proof of Hypothesis 3.1 in $\S5$ of \cite{cs1} is based on having a  Poincar\'{e} inequality.  In particular, the proof of Proposition 5.3 of \cite{cs1}  goes through for the Bernoulli environment, but with some caveats.  The inequality (\ref{AB5}) above 
implies the inequality (5.28) of \cite{cs1} with $m=1$. However the gradient field equation   (5.29) of \cite{cs1} does not hold in the precise form. The reason is that the Leibniz product rule for derivatives does not hold for the functions $F(Y(\cdot))$ 
we are considering. We have $F(Y(\cdot))=a+bY_0$, where $a,b$ depend only on the variables $Y(x), \ x\ne 0$.  Observe now that
\begin{multline} \label{AC5}
Y_0F(Y(\cdot)) \ = \ b+aY_0, \quad \frac{\pa}{\pa Y_0} [Y_0F(Y(\cdot))] \ = \ a \ , \quad Y_0\frac{\pa F(Y(\cdot))}{\pa Y_0} \ = \ bY_0 \ , \\
\frac{\pa}{\pa Y_0} [Y_0F(Y(\cdot))] \ = \   Y_0\frac{\pa F(Y(\cdot))}{\pa Y_0} + F^R(Y(\cdot)) \ , \quad F^R(Y(\cdot))  \ = \ a-bY_0 \ .
\end{multline}
We consider (5.17) of \cite{cs1}, which for the Bernoulli environment is given by
\be \label{AD5}
\eta\Phi(\xi,\eta,Y(\cdot))+ \pa_\xi^*[1+\ga Y_0]\pa_\xi\Phi(\xi,\eta,Y(\cdot)) \ = \ -\pa^*_\xi [1+\ga Y_0]I_d \ .
\ee
We take $F(Y(\cdot))=\Phi(\xi,\eta,Y(\cdot))v$ for any $v\in\mathbb{C}^d$ and apply the gradient operator to (\ref{AD5}) acting on  $v$. We obtain the equation
\be \label{AE5}
\eta dF(\cdot,Y(\cdot))+ D_\xi^*D_\xi F(Y(\cdot)) +\ga D^*_\xi d[Y_0\pa_\xi F(Y(\cdot))] \ = \ -D^*_\xi \ga \del(\cdot)v \ .
\ee
It follows from (\ref{AC5}) that
\be \label{AF5}
d[Y_0\pa_\xi F(Y(\cdot))] \ = \  Y_0D_\xi dF(\cdot,Y(\cdot))+\del(\cdot) [\pa_\xi F(Y(\cdot))]^R \ .
\ee
We conclude from (\ref{AE5}), (\ref{AF5}) that the equation corresponding to (5.29) of \cite{cs1} is
\be \label{AG5}
\eta dF(\cdot,Y(\cdot))+ D_\xi^*[1+\ga Y_0]D_\xi F(Y(\cdot)) \ = \ -D^*_\xi \ga \del(\cdot)\{v+[\pa_\xi F(Y(\cdot))]^R]\} \ . 
\ee
Since $\|[\pa_\xi F(Y(\cdot))]^R]\|_{L^2(\Om;\mathbb{C}^d)}=\|\pa_\xi F(Y(\cdot))]\|_{L^2(\Om;\mathbb{C}^d)}$, the remainder of the argument of Proposition 5.3 of \cite{cs1} goes through as before.
\end{proof}
\vspace{.1in}
\section{Euclidean  Massless Gradient  Field  Environment}
Hypothesis 3.1 was already proved in \cite{cs1} for the massless field environment. To prove  Hypothesis 4.1 we use the alternative method for generating the massless field measure which was set out in \cite{cd}. 
First we derive a covariance formula analogous to the Helffer-Sj\"{o}strand formula \cite{hs}. We are interested in functions of the gradient field $\om(x)=\na\phi(x), \ x\in\mathbb{Z}^d$,  where $\phi:\mathbb{Z}^d\ra\mathbb{R}$. 
We associate with two functions $F(\om(\cdot)), \ G(\om(\cdot))$ a function $q(\al,\beta), \ \al,\beta\in\mathbb{R}$ defined by 
\be \label{A6}
q(\al,\beta) \ = \ -\log\left\{\left\langle      \exp\left[-\al  F(\om(\cdot))-\beta G(\om(\cdot))  \right]           \right\rangle \right\} \ , \quad \al,\beta\in\mathbb{R} \ .
\ee
Differentiating (\ref{A6}) we obtain the formulae,
\be \label{B6}
\begin{array}{ccc}
\frac{\pa q(\al,\beta)}{\pa\al} \ &=& \  \left\langle F(\om(\cdot))\right\rangle_{\al,\beta} \ , \\
\frac{\pa^2 q(\al,\beta)}{\pa\al\pa\beta} \ &=& \  -{\rm cov}_{\al,\beta} [F(\om(\cdot)), G(\om(\cdot))] \ , 
\end{array}
\ee
where $\langle\cdot\rangle_{\al,\beta}$ and ${\rm cov}_{\al,\beta}[\cdot,\cdot]$ denote expectation and covariance with respect to the probability measure on fields $\phi:\mathbb{Z}^d\ra\mathbb{R}$  with density proportional to the function
\be \label{C6}
\phi(\cdot) \ \ra \ \exp\left[  -\al F(\na(\phi(\cdot))-\beta G(\na\phi(\cdot))  -\sum_{x\in\mathbb{Z}^d} V(\na\phi(x))+m^2\phi(x)^2                                  \right]  \ .
\ee

Following the argument of $\S4$ of \cite{cd} we observe that the  probability measure corresponding to (\ref{C6}) is the invariant measure for the stochastic dynamics of a function $\phi_{\al,\beta}(x,t), \ x\in\mathbb{Z}^d, \ t\ge 0$.
The dynamical equation of motion is given by 
\begin{multline} \label{D6}
d\phi_{\al,\beta}(x,t) \ =  \ -\frac{1}{2}\Big[     \phi_{\al,\beta}(x,t) +[-\De+m^2]^{-1}\na^*\big\{\al D_{\om(x)}F(\na\phi_{\al,\beta}(\cdot,t))+\beta D_{\om(x)}G(\na\phi_{\al,\beta}(\cdot,t)) \\
+V'(\na\phi_{\al,\beta}(x,t))-\na\phi_{\al,\beta}(x,t)\big\}                      \Big]dt+[-\De+m^2]^{-1/2}dB(x,t) \ , \quad x\in\mathbb{Z}^d \ ,
\end{multline}
where $B(x,\cdot),  \ x\in\mathbb{Z}^d,$  are independent Brownian motions. Letting $\om_{\al,\beta}(\cdot,t)=\na \phi_{\al,\beta}(\cdot,t)$ we obtain from (\ref{D6}) a stochastic evolution equation for  $t\ra  \om_{\al,\beta}(\cdot,t)$,
\begin{multline} \label{E6}
d\om_{\al,\beta}(x,t) \ =  \ -\frac{1}{2}\Big[     \om_{\al,\beta}(x,t) +\na[-\De+m^2]^{-1}\na^*\big\{\al D_{\om(x)}F(\om_{\al,\beta}(\cdot,t))+\beta  D_{\om(x)}G(\om_{\al,\beta}(\cdot,t)) \\
+V'(\om_{\al,\beta}(x,t))-\om_{\al,\beta}(x,t)\big\}                       \Big]dt+\na[-\De+m^2]^{-1/2}dB(x,t) \ , \quad x\in\mathbb{Z}^d \ .
\end{multline}
We then have as in (4.4) of \cite{cd} that
\be \label{F6}
\lim_{T\ra\infty} \left\langle F(\om_{\al,\beta} (\cdot,T))\right\rangle \ = \ \left\langle F(\om(\cdot))\right\rangle_{\al,\beta} \ .
\ee
It follows from (\ref{B6}), (\ref{F6}) that
\be \label{G6}
\frac{\pa q(\al,\beta)}{\pa\al} \ = \ \lim_{T\ra\infty} \left\langle F(\om_{\al,\beta} (\cdot,T))\right\rangle \ .
\ee
Differentiating (\ref{G6}) with respect to $\beta$ and using (\ref{B6}) we conclude that 
\be \label{H6}
{\rm cov}_{\al,\beta} [F(\om(\cdot)), G(\om(\cdot))] \ = \  -\lim_{T\ra\infty} \sum_{x\in\mathbb{Z}^d} \left\langle \left[D_{\om(x)} F(\om_{\al,\beta} (\cdot,T)), \  \frac{\pa \om_{\al,\beta} (x,T)}{\pa\beta}\right]_{\mathbb{R}^d} \right\rangle \ .
\ee
To obtain a formula for the field covariance of $F(\om(\cdot))$ and $G(\om(\cdot))$ we set $\al=\beta=0$ in (\ref{H6}). We may obtain an evolution equation for the function  $t\ra  \pa\om_{0,0}(\cdot,t)/\pa\beta$ by differentiating (\ref{E6}) with respect to 
$\beta$ and setting $\al=\beta=0$. The resulting evolution equation is given by 
\begin{multline} \label{I6}
d\left\{\frac{\pa \om_{0,0}(x,t)}{\pa\beta} \right\} \ =  \Bigg[ -\frac{1}{2}  \frac{\pa\om_{0,0}(x,t)}{\pa\beta} \\
 +\frac{1}{2}\na[-\De+m^2]^{-1}\na^*\left\{   
I_d-V''(\om_{0,0}(x,t)) \right\}  \frac{\pa\om_{0,0}(x,t)}{\pa\beta} -\frac{1}{2}D_{\om(x)}G(\om_{0,0}(\cdot,t))\Bigg]dt \ , \quad x\in\mathbb{Z}^d \ .
\end{multline}

Note that (\ref{I6}), being a first variation equation of a stochastic differential equation, is linear.  We also have that $I_d-V''(\om)=\mathbf{b}(\om), \  \om\in\mathbb{R}^d$, where the symmetric $d\times d$ matrix $\mathbf{b}(\cdot)$ satisfies the quadratic form inequality 
$0\le \mathbf{b}(\cdot)\le (1-\la)I_d$.  Let $\mathcal{L}_T$ be the linear operator on functions $f:\mathbb{Z}^d\times [0,T]\ra\mathbb{R}^d$ defined by
\be \label{J6}
\mathcal{L}_Tf(\cdot,t) \ = \ \frac{1}{2}\int_0^t e^{_(t-s)/2}  \ ds \  \na[-\De+m^2]^{-1}\na^*\mathbf{b}( \om_{0,0}(\cdot,s))f(\cdot,s) \ .
\ee
Then $f(\cdot,t)= -\pa\om_{0,0}(\cdot,t)/\pa\beta$ is a solution to the affine equation,
\be \label{K6}
f(\cdot,t) \ = \ \mathcal{L}_T f(\cdot,t) + \frac{1}{2}\int_0^t e^{-(t-s)/2} g(\cdot,s) \ ds  \ ,  \quad g(\cdot,s) \ = \ D_{\om(\cdot)}G(\om_{0,0}(\cdot,s)) \ .
\ee
\begin{lem}
The following covariance formula holds: 
\begin{multline} \label{L6}
{\rm cov} [F(\om(\cdot)), G(\om(\cdot))] \ = \
\sum_{x,x'\in\mathbb{Z}^d} \left\langle \left[D_{\om(x)} F(\om (\cdot)), \  K(x,x',\om(\cdot))D_{\om(x')} G(\om (\cdot))\right]_{\mathbb{R}^d} \right\rangle \ ,
\end{multline}
where the kernel $(x,x')\ra K(x,x',\om(\cdot))\in\mathbb{R}^d\otimes \mathbb{R}^d$ satisfies an inequality $|K(x,x',\om(\cdot))v|\le C|v|/[1+|x-x'|^k], \ v\in\mathbb{R}^d,  \ x,x'\in\mathbb{Z}^d,$
for some $k>0$ and  constant $C$ depending only on $\la,d$. 
\end{lem}
\begin{proof}
We follow the argument in the proof of Proposition 4.1 of \cite{cd}.  Thus we have from (\ref{H6})-(\ref{K6})  that
\begin{multline} \label{M6}
{\rm cov} [F(\om(\cdot)), G(\om(\cdot))] \ = \  \lim_{T\ra\infty} \frac{1}{2}\int_0^T dt \ e^{-(T-t)/2} \ \times \\
\sum_{x,x'\in\mathbb{Z}^d} \left\langle \left[D_{\om(x)} F(\om_{0,0} (\cdot,T)), \  K(x,x',t,T,\om_{0,0}(\cdot,\cdot))D_{\om(x')} G(\om_{0,0} (\cdot,t))\right]_{\mathbb{R}^d} \right\rangle \ ,
\end{multline}
where the function $K$ in (\ref{M6})  an be obtained from the function  $a_n(\cdot,t,T)$ defined by (4.22) of \cite{cd}  as $n\ra\infty$.  Let $T_{K,t,T}$ be the operator on $\ell_2(\mathbb{Z}^d;\mathbb{R}^d)$
defined by  
\be \label{N6}
T_{K,t,T} g(x') \ = \ \sum_{x\in\mathbb{Z}^d}  g(x)^* K(x,x',t,T,\om_{0,0}(\cdot,\cdot)) \ , \quad x'\in\mathbb{Z}^d \ .
\ee
One sees from the formula (4.22) of \cite{cd} that $T_{K,t,T}$ is a  bounded operator on $\ell_2(\mathbb{Z}^d;\mathbb{R}^d)$ with norm satisfying $\left\|         T_{K,t,T}            \right\|_{\ell_2(\mathbb{Z}^d;\mathbb{R}^d)}\le\exp[(1-\la)(T-t)/2]$. 
We also see  upon using  Theorem 5.1 that $T_{K,t,T}$ is a  bounded operator on $\ell_{p,k}(\mathbb{Z}^d;\mathbb{R}^d)$ with norm satisfying $\left\|         T_{K,t,T}            \right\|_{\ell_2(\mathbb{Z}^d;\mathbb{R}^d)}\le\exp[\kappa_{p,k}(1-\la)(T-t)/2]$. 
Since $\kappa_{2,0}=1$ we conclude from continuity of norms that there is a $k>0$ such that $\kappa_{2,k}(1-\la)<1$. In that case the RHS of (\ref{M6}) remains bounded as $T\ra\infty$.  Since $\om(\cdot,T)$ converges in distribution to $\om(\cdot)$ as $T\ra\infty$ we obtain the formula (\ref{L6}), with the kernel $K$ satisfying the inequality in the statement of the Lemma. 
\end{proof}
\begin{proof}[Proof of Hypothesis 4.1]
We proceed similarly to the proof of Lemma 5.3, now using the covariance formula (\ref{L6}). The main difference from the proof of Hypothesis 4,1 in $\S5$ is that the formulae (\ref{X5})-(\ref{AA5}) need to be slightly altered since now $y_{j_1}=y, 
y_{j_2}=y'$  and we no longer have the constraint $y=y'$.  Using Lemma  6.1 the inequality (\ref{AA5}) is replaced by
\begin{multline} \label{O6}
\sum_{y,y'\in\mathbb{Z}^d} \frac{|f_1(y)||f_2(y',x)|}{[1+|y-y'|^k]} \\
 \le \   \sum_{y,y'\in\mathbb{Z}^d}\frac{C_dm^{2d+1}(1-\la)^{j_1+m-j_2}\{\kappa_{p,0}+\kappa\}^{j_1+m-j_2}}{[1+|y|]^{2d/p-d}[1+|y-y'|]^k[1+|y'-x|]^{2d/p-d}} \ .
\end{multline}
On choosing $p$ in (\ref{O6}) close to $1$ we see that the sum on the RHS is bounded by  $C^mC_\ve/[1+|x|]^{k-\ve}$ for any $\ve>0$, where the constant $C_\ve$ depends on $\ve$ and diverges as $\ve\ra0$. 
\end{proof}

\vspace{.1in}

\appendix

\section{On Bourgain and Kim-Lemm estimates for  $\hat{h}_m(\cdot,\eta)$}
In \cite{b} Bourgain proved Hypothesis 4.1 for the i.i.d. case, showing that (\ref{AB4}) holds  for any value of $\al<d$. This result was improved by Kim and Lemm \cite{kl},
who showed that in the i.i.d. case (\ref{AB4}) holds for any value of $\al<2d$.  Their methodology is essentially a further iteration of Bourgain's method.  Here we illustrate graphically
Bourgain's method,  the Kim-Lemm improvement of Bourgain's result, and why no further improvement is possible.

In the graphs below we illustrate  Bourgain's proof of Lemma 5.2 (Lemma 1 of \cite{b}), and then the proof of Hypothesis 4.1 for $\al<d$ in (\ref{AB4}). The first graph  illustrates the proof of Lemma 5.2. Thus we begin with the simplest of connected 
 tree graphs, on $m+1$ vertices with $m$ edges.  The highlighted edge $[\tau-1,\tau]$ corresponds to the gap $|y_\tau-y_{\tau-1}|\ge 2^nR$. On removing this edge we are left with two connected components of  tree graphs, one containing $y_0=0$ and the other $y_m=x$. 
 Using the approximate uniform summability of the kernel $\na\na^*G_\eta(\cdot)$ for $\eta>0$,   we may contract these graphs down to $y_0$ and $y_m$ respectively at a cost of an exponential factor in the constant. Thus  we obtain an inequality
 $|\hat{h}_m(x,\om,\eta)|\le C_\al^m/(1+|x|)^\al$ for any $\al<d$. 
 
 The second graph below illustrates Bourgain's proof of Hypothesis 4.1 for $\al<d$.  The edge $[\tau-1,\tau]$ has been removed from the first graph, leaving a disconnected graph with two components. These are now reconnected by using the covariance
 formula (\ref{X5}).  Thus we  have a connected graph on the  vertices $0,1,\dots,j_1,j_2,\dots,m$. Hence there is an edge $[\tau_1-1,\tau_1]$ on this graph with $|y_{\tau_1}-y_{\tau_1-1}|\simeq R$.   Removing this edge, we obtain two connected components of  tree graphs, one containing the vertex $y_0$ and the other the vertex $y_m$. We may now contract these as before, yielding the bound (\ref{AB4}) with $\al<d$. 

\begin{tikzpicture}[thick]
\draw[red][-] (5,0)--(6,0);
\fill (0,0) circle(2pt);
\fill (12,0) circle(2pt);
\fill (0.5,0) circle(2pt);
\fill (1,0) circle(2pt);
\fill (5,0) circle(2pt);
\fill (6,0) circle(2pt);
\draw[-] (0,0)--(5,0);
\draw[-] (6,0)--(12,0);
\draw (0,0) node[anchor=north] {$0$};
\draw (0.5,0) node[anchor=north] {$1$};
\draw (12,0) node[anchor=north] {$m$};
\draw (1,0) node[anchor=north] {$2$};
\draw (5,0) node[anchor=north] {$\tau-1$};
\draw (6,0) node[anchor=north] {$\tau$};
\draw (5,1.5) node[anchor=north] {\bf Bourgain Estimate of   $\hat{h}_m(x,\om,\eta)$};
\draw (5,-1.5) node[anchor=north] {\bf Bourgain Estimate of   $\hat{h}_m(x,\eta)$};
\fill (0,-3) circle(2pt);
\fill (12,-3) circle(2pt);
\fill (0.5,-3) circle(2pt);
\fill (1,-3) circle(2pt);
\fill (5,-3) circle(2pt);
\fill (6,-3) circle(2pt);
\draw[-] (3.5,-3)--(4,-3);
\draw[-] (0,-3)--(2.5,-3);
\draw[-] (3.5,-3)--(5,-3);
\draw[-] (6,-3)--(12,-3);
\fill (2.5,-3) circle(2pt);
\fill (3.5,-3) circle(2pt);
\draw (0,-3) node[anchor=north] {$0$};
\draw (0.5,-3) node[anchor=north] {$1$};
\draw (12,-3) node[anchor=north] {$m$};
\draw (1,-3) node[anchor=north] {$2$};
\draw (5,-3) node[anchor=north] {$\tau-1$};
\draw (6,-3) node[anchor=north] {$\tau$};
\fill (4,-3) circle(2pt);
\fill (8,-3) circle(2pt);
\draw (4,-3) node[anchor=north] {$j_1$};
\draw (8,-3) node[anchor=north] {$j_2$};
\draw[blue] (8,-3) arc (60:120: 4 cm);
\draw (2.5,-3) node[anchor=north] {$\tau_1-1$};
\draw (3.5,-3) node[anchor=north] {$\tau_1$};
\draw[red][-] (2.5,-3)--(3.5,-3);

\end{tikzpicture}

Note that the second graph above is an iteration of Bourgain's method for the first graph. Kim-Lemm \cite{kl} carried out a further iteration of the second graph, as illustrated below:

\begin{tikzpicture}[thick]

\draw (5,-2.5) node[anchor=north] {\bf Kim-Lemm Estimate of   $\hat{h}_m(x,\eta)$};
\fill (0,-4) circle(2pt);
\fill (12,-4) circle(2pt);
\fill (0.5,-4) circle(2pt);
\fill (1,-4) circle(2pt);
\fill (5,-4) circle(2pt);
\fill (6,-4) circle(2pt);
\draw[-] (3.5,-4)--(4,-4);
\draw[-] (0,-4)--(2.5,-4);
\draw[-] (3.5,-4)--(5,-4);
\draw[-] (6,-4)--(12,-4);
\fill (2.5,-4) circle(2pt);
\fill (3.5,-4) circle(2pt);
\draw (0,-4) node[anchor=north] {$0$};
\draw (0.5,-4) node[anchor=north] {$1$};
\draw (12,-4) node[anchor=north] {$m$};
\draw (1,-4) node[anchor=north] {$2$};
\draw (5,-4) node[anchor=north] {$\tau-1$};
\draw (6,-4) node[anchor=north] {$\tau$};
\fill (4,-4) circle(2pt);
\fill (8,-4) circle(2pt);
\draw (4,-4) node[anchor=north] {$j_1$};
\draw (8,-4) node[anchor=north] {$j_2$};
\draw[blue] (8,-4) arc (60:120: 4 cm);
\draw (2.5,-4) node[anchor=north] {$\tau_1-1$};
\draw (3.5,-4) node[anchor=north] {$\tau_1$};
\draw[blue] (4.5,-4) arc (60:120: 3 cm);
\end{tikzpicture}

The edge $[\tau_1-1,\tau_1]$ is removed from the second  of the Bourgain graphs. Then, on using the covariance formula (\ref{X5}), a new edge is introduced linking the two connected components. Thus we obtain another  connected tree graph
on the  $m+1$ vertices, yielding a proof of Hypothesis 4.1 with $\al$ in (\ref{AB4}) up to (but not including)  $2d$.

Finally we show why there can be no further iteration of Bourgain's method.  We have already observed that the Kim-Lemm graph is connected, whence there is an edge $[\tau_2-1,\tau_2]$ such that $|y_{\tau_2}-y_{\tau_2-1}|\simeq R$. 
On removing this edge and contracting the resulting two connected components of the graph we obtain (\ref{AB4})  for $\al<2d$. Using the covariance formula (\ref{X5}) we may introduce another edge in the graph with  vertices at nodes $r_1,r_2$
with $r_1\le\tau_2-1$ and $r_2\ge \tau_2$.  We have illustrated this below  with the edge in red. Observe that the resulting graph has still two connected components, with $0$ in one component and $m$ in the other. 
The reason  is that this further iteration of Bourgain's method  can result in a graph containing  a cycle. Therefore  Bourgain's method cannot be used to improve $\al$ in (\ref{AB4}) beyond $2d$. One can see by direct calculation in the case of a Bernoulli environment that (\ref{AB4}) cannot be improved beyond $\al<2d$. 

\begin{tikzpicture}[thick]
\draw (5,-2.5) node[anchor=north] {\bf Creating a Cycle};
\fill (0,-4) circle(2pt);
\fill (12,-4) circle(2pt);
\fill (8,-4) circle(2pt);
\fill (9,-4) circle(2pt);
\draw[-] (3.5,-4)--(5,-4);
\draw[-] (0,-4)--(2.5,-4);
\draw[-] (6,-4)--(8,-4);
\draw[-] (9,-4)--(12,-4);
\fill (2.5,-4) circle(2pt);
\fill (3.5,-4) circle(2pt);
\draw (0,-4) node[anchor=north] {$0$};
\draw (12,-4) node[anchor=north] {$m$};
\draw (8,-4) node[anchor=north] {$\tau-1$};
\draw (9,-4) node[anchor=north] {$\tau$};
\fill (4,-4) circle(2pt);
\fill (10,-4) circle(2pt);
\draw (4,-4) node[anchor=north] {$j_1$};
\draw (10,-4) node[anchor=north] {$j_2$};
\draw[blue] (10,-4) arc (60:120: 6 cm);
\draw (2.5,-4) node[anchor=north] {$\tau_1-1$};
\draw (3.5,-4) node[anchor=north] {$\tau_1$};
\draw[blue] (10,-4) arc (60:120: 6 cm);
\draw (5,-4) node[anchor=north] {$\tau_2-1$};
\draw (6,-4) node[anchor=north] {$\tau_2$};
\fill (5,-4) circle(2pt);
\fill (6,-4) circle(2pt);
\draw[blue] (7.5,-4) arc (60:120: 6 cm);
\draw[red] (1,-4) arc (240:300: 6 cm);
\end{tikzpicture}

\end{document}